\documentclass[a4,10pt]{article}
\usepackage{amsmath}
\usepackage{amssymb}
\usepackage{amsthm}
\usepackage{graphicx}
\usepackage{amscd}
\usepackage{ulem}
\usepackage{makeidx}
\usepackage{fancybox}

\newcommand{\ccc}{{\mathbf C}}

\newcommand{\nnn}{{\mathbf N}}

\newcommand{\zzz}{{\mathbf Z}}

\newcommand{\hhh}{{\frak{h}}}

\newtheorem{prop}{Proposition}[section]
\newtheorem{lemma}{Lemma}[section]
\newtheorem{cor}{Corollary}[section]

\newtheorem{note}{Note}[section]

\numberwithin{equation}{section}

\begin{document}

\title{
Mock theta functions and characters of N=3 superconformal modules II}

\author{\footnote{12-4 Karato-Rokkoudai, Kita-ku, Kobe 651-1334, Japan, \,\ 
wakimoto@r6.dion.ne.jp, \,\ wakimoto.minoru.314@m.kyushu-u.ac.jp
}{ Minoru Wakimoto}}

\date{\empty}

\maketitle

\begin{center}
Abstract
\end{center}
We obtain an explicit formula for the mock theta function 
$\Phi^{[m,s]}$ in the case when either $m$ or $s$ is a half of an odd integer
by using the coroot lattice of $D(2,1;a)$. 
This enables us, together with the recurrence formula for $\Phi^{[m,s]}$,
to know $\Phi^{[m,s]}$ for all $m$ and $s$.
As its application, we deduce explicit formulas for the character 
and the supercharacter of N=3 modules obtained from the quantum 
Hamiltonian reduction of $\widehat{osp}(3|2)$-modules.

\tableofcontents

\section{Introduction}

In this paper, as continuation from \cite{W2022}, we consider 
characters of N=3 modules $H(\Lambda^{[K(m),m_2]})$ obtained from 
the quantum Hamiltonian reduction of the highest weight 
$\widehat{B}(1,1)$-module $L(\Lambda^{[K(m),m_2]})$ where 
$m$ is a positive integer and $m_2$ is a non-negative integer 
such that $m_2 \leq m$ and 
\begin{equation}
K(m) := -\frac{m+2}{4} \hspace{5mm} \text{and} \hspace{5mm}
\Lambda^{[K(m),m_2]} := K(m) \Lambda_0-\frac{m_2}{2}\alpha_1.
\label{n3-2:eqn:2022-319a}
\end{equation}
Very important functions to describe the characters are the mock theta 
functions $\Phi^{[m,s]}_i$ $(i=1,2)$ and $\Phi^{[m,s]}$, where 
$m \in \frac12 \nnn$ and $s \in \frac12 \zzz$, defined by 
\begin{subequations}
\begin{eqnarray}
\Phi^{[m,s]}_1(\tau, z_1, z_2,t) &:=& e^{-2\pi imt} 
\sum_{j \in \zzz}
\frac{e^{2\pi imj(z_1+z_2)+2\pi isz_1} q^{mj^2+sj}}{1-e^{2\pi iz_1}q^j}
\label{n3-2:eqn:2022-111a}
\\[1mm]
\Phi^{[m,s]}_2(\tau, z_1, z_2,t) &:=& e^{-2\pi imt} 
\sum_{j \in \zzz}
\frac{e^{-2\pi imj(z_1+z_2)-2\pi isz_2} q^{mj^2+sj}}{1-e^{-2\pi iz_2}q^j}
\label{n3-2:eqn:2022-111b}
\\[1mm]
\Phi^{[m,s]}(\tau, z_1, z_2,t) &:=& 
\Phi^{[m,s]}_1(\tau, z_1, z_2,t) -\Phi^{[m,s]}_2(\tau, z_1, z_2,t) 
\label{n3-2:eqn:2022-111c}
\end{eqnarray}
\end{subequations}
where $q:=e^{2\pi i\tau} \,\ (\tau \in \ccc_+)$ and $z_1, z_2, t \in \ccc$. 
In section 2, we prepare some of basic properties of these functions 
which are used in this paper.

The main part of the present paper is section 3, in which we deduce a 
relation between $\Phi^{[m,0]}$ and $\Phi^{[1,0]}$ 
by using the coroot lattice of $D(2,1;a)$.
This enables us to obtain an explicit formula (Corollary 
\ref{n3-2:cor:2022-314a}) for $\Phi^{[m,s]}$ 
$(m \in \frac12 \nnn, \,\ s \in \frac12 \zzz_{\rm odd})$, 
since $\Phi^{[1,0]}$ is a well-known function.
Using this formula and basic properties of $\Phi^{[m,s]}$,
we get a formula (Proposition \ref{n3-2:prop:2022-331a})
for $\Phi^{[m,s]}$ ($m \in \frac12 \nnn_{\rm odd}, \,\ s \in \zzz)$.
By these formulas and the recurrence formula, 
we can know $\Phi^{[m,s]}$ for all $m$ and $s$ 
(Proposition \ref{n3-2:prop:2022-403b}).

In section 4, using these results, we deduce a formula for 
the character and the supercharacter of the N=3 module 
$H(\Lambda^{[K(m),m_2]})$.

\section{Basic properties of $\Phi^{[m,s]}$}

Some of basic properties of  $\Phi^{[m,s]}_i$ and $\Phi^{[m,s]}$ are 
shown in section 2 of \cite{W2022}. In this section we add some more 
which are used in this paper.

\begin{lemma} 
\label{n3-2:lemma:2022-312a}
Let $m \in \frac12 \nnn$, $s \in \frac12 \zzz$ and $j \in \zzz$. Then 
\begin{enumerate}
\item[{\rm 1)}] if $j \geq 0$, 
\begin{enumerate}
\item[{\rm (i)}] \,\ $
\Phi^{[m, s+j]}_1(\tau, z_1,z_2,0) \, = \, \Phi^{[m, s]}_1(\tau, z_1,z_2,0)$
$$
- \, \sum_{k=0}^{j-1}e^{\pi i(s+k)(z_1-z_2)} q^{-\frac{(s+k)^2}{4m}}
\theta_{s+k,m}(\tau, z_1+z_2)
$$
\item[{\rm (ii)}] \,\ $
\Phi^{[m, s+j]}_2(\tau, z_1,z_2,0) \, = \, \Phi^{[m, s]}_2(\tau, z_1,z_2,0)$
$$
- \, \sum_{k=0}^{j-1}e^{\pi i(s+k)(z_1-z_2)} q^{-\frac{(s+k)^2}{4m}}
\theta_{-(s+k),m}(\tau, z_1+z_2)
$$
\item[{\rm (iii)}] \,\ $
\Phi^{[m, s+j]}(\tau, z_1,z_2,0) \, = \, \Phi^{[m, s]}(\tau, z_1,z_2,0)$
$$
- \, \sum_{k=0}^{j-1}e^{\pi i(s+k)(z_1-z_2)} q^{-\frac{(s+k)^2}{4m}}
\big[\theta_{s+k,m}-\theta_{-(s+k),m}\big](\tau, z_1+z_2)
$$
\end{enumerate}
\item[{\rm 2)}] \,\ if $j < 0$, 
\begin{enumerate}
\item[{\rm (i)}] \,\ $
\Phi^{[m, s+j]}_1(\tau, z_1,z_2,0) \, = \, \Phi^{[m, s]}_1(\tau, z_1,z_2,0)$
$$
+ \, \sum_{k=j}^{-1}e^{\pi i(s+k)(z_1-z_2)} q^{-\frac{(s+k)^2}{4m}}
\theta_{s+k,m}(\tau, z_1+z_2)
$$
\item[{\rm (ii)}] \,\ $
\Phi^{[m, s+j]}_2(\tau, z_1,z_2,0) \, = \, \Phi^{[m, s]}_2(\tau, z_1,z_2,0)$
$$
+ \, \sum_{k=j}^{-1}e^{\pi i(s+k)(z_1-z_2)} q^{-\frac{(s+k)^2}{4m}}
\theta_{-(s+k),m}(\tau, z_1+z_2)
$$
\item[{\rm (iii)}] \,\ $
\Phi^{[m, s+j]}(\tau, z_1,z_2,0) \, = \, \Phi^{[m, s]}(\tau, z_1,z_2,0)$
$$
+ \, \sum_{k=j}^{-1}e^{\pi i(s+k)(z_1-z_2)} q^{-\frac{(s+k)^2}{4m}}
\big[\theta_{s+k,m}-\theta_{-(s+k),m}\big](\tau, z_1+z_2)
$$
\end{enumerate}
\end{enumerate}
where 
$$
\theta_{k,m}(\tau,z) := \sum_{j \in \zzz}
e^{2\pi im(j+\frac{k}{2m})z}q^{m(j+\frac{k}{2m})^2}
$$ 
is the Jacobi's theta function.
\end{lemma}

\begin{proof} 1) is due to Lemma 2.8 in \cite{W2022}. 
2) follows easily from 1).
\end{proof}

Note that, in particular, we have 
\begin{equation}
\Phi^{[m, 0]}_i(\tau, z_1,z_2,t) \, = \, \Phi^{[m, 1]}_i(\tau, z_1,z_2,t)
\hspace{5mm} {\rm for} \quad  m \in \tfrac12 \nnn \,\ 
\text{and} \,\ i \in \{1,2\}.
\label{n3-2:eqn:2022-331d}
\end{equation}

From this lemma, we obtain the following Lemma \ref{n3-2:lemma:2022-320b},
which will be used to prove Proposition \ref{n3-2:prop:2022-320b} 
in section 4.

\begin{lemma}
\label{n3-2:lemma:2022-320b}
For $m \in \nnn$, the following formulas hold:
\begin{enumerate}
\item[{\rm 1)}] \,\ If $s \in \frac12 \nnn_{\rm odd}$, 
\begin{enumerate}
\item[{\rm (i)}] $\Phi^{[\frac{m}{2}, s]}\Big(2\tau, \, 
z+\dfrac{\tau}{2}, \, z-\dfrac{\tau}{2}, \, 0\Big) - 
\Phi^{[\frac{m}{2}, \frac12]}
\Big(2\tau, \, z+\dfrac{\tau}{2}, \, z-\dfrac{\tau}{2}, \, 0\Big)$
$$
= \,\ \sum_{k=0}^{s-\frac32}
q^{-\frac{1}{m}(k+\frac12)^2 + \frac12 (k+\frac12)} \, 
\big[\theta_{k+\frac12, \frac{m}{2}}-\theta_{-(k+\frac12), \frac{m}{2}}\big]
(2\tau, 2z)
$$

\item[{\rm (ii)}] $\Phi^{[\frac{m}{2}, s]}\Big(2\tau, \, 
z+\dfrac{\tau}{2}-\dfrac12, \, z-\dfrac{\tau}{2}+\dfrac12, \, 0\Big) 
- 
\Phi^{[\frac{m}{2}, \frac12]}
\Big(2\tau, \, z+\dfrac{\tau}{2}-\dfrac12, \, 
z-\dfrac{\tau}{2}+\dfrac12, \, 0\Big)$
$$
= \,\ - \, i \sum_{k=0}^{s-\frac32}
(-1)^k \, q^{-\frac{1}{m}(k+\frac12)^2 + \frac12 (k+\frac12)} \, 
\big[\theta_{k+\frac12, \frac{m}{2}}-\theta_{-(k+\frac12), \frac{m}{2}}\big]
(2\tau, 2z)
$$
\end{enumerate}
\item[{\rm 2)}] \,\ If $s \in \nnn$, 
\begin{enumerate}
\item[{\rm (i)}] $\Phi^{[\frac{m}{2}, s]}\Big(2\tau, \, 
z+\dfrac{\tau}{2}, \, z-\dfrac{\tau}{2}, \, 0\Big) - 
\Phi^{[\frac{m}{2}, 0]}
\Big(2\tau, \, z+\dfrac{\tau}{2}, \, z-\dfrac{\tau}{2}, \, 0\Big)$
$$
= \,\ \sum_{k=0}^{s-1}
q^{-\frac{k^2}{m} + \frac{k}{2}} \, 
\big[\theta_{k, \frac{m}{2}}-\theta_{-k, \frac{m}{2}}\big]
(2\tau, 2z)
$$

\item[{\rm (ii)}] $\Phi^{[\frac{m}{2}, s]}\Big(2\tau, \, 
z+\dfrac{\tau}{2}-\dfrac12, \, z-\dfrac{\tau}{2}+\dfrac12, \, 0\Big) 
- 
\Phi^{[\frac{m}{2}, 0]}
\Big(2\tau, \, z+\dfrac{\tau}{2}-\dfrac12, \, 
z-\dfrac{\tau}{2}+\dfrac12, \, 0\Big)$
$$
= \,\ - \, i \sum_{k=0}^{s-1}
(-1)^k \, q^{-\frac{k^2}{m} + \frac{k}{2}} \, 
\big[\theta_{k, \frac{m}{2}}-\theta_{-k, \frac{m}{2}}\big]
(2\tau, 2z)
$$
\end{enumerate}
\end{enumerate}
\end{lemma}

\begin{lemma} 
\label{n3-2:lemma:2022-313b}
Let $m \in \frac12 \nnn$, \, $s \in \frac12 \zzz$ and $p \in \zzz$. 
Then the following formulas hold for $i, j \in \{1,2\}$ such that 
$i \ne j \, :$ 
\begin{enumerate}
\item[{\rm 1)}] $\Phi^{[m,s]}_i(\tau, \, z_1, \, z_2+p\tau, \, t) 
=
e^{-2\pi imp(z_1+z_2)} \, 
\Phi^{[m,s]}_i(\tau, \, z_1-p\tau, \, z_2, \, t)$

\item[{\rm 2)}] $\Phi^{[m,s]}_i(\tau, \, z_1, \, z_2+p\tau, \, t) 
=
e^{-2\pi imp(z_1+z_2)} \, 
\Phi^{[m,s]}_j(\tau, \, -z_2, \, -z_1+p\tau, \, t) $
\end{enumerate}
\end{lemma}

\begin{proof} 1) First we prove in the case $i=1$. By definition 
\eqref{n3-2:eqn:2022-111a} of $\Phi^{[m,s]}_1$, 
{\allowdisplaybreaks
\begin{eqnarray*}
& & \hspace{-6mm}
\Phi^{[m,s]}_1(\tau, \, z_1, \, z_2+p\tau, \, t) \, = \, 
e^{-2\pi imt} 
\sum_{j \in \zzz}
\frac{e^{2\pi imj(z_1+z_2+p\tau)+2\pi isz_1} q^{mj^2+sj}}{1-e^{2\pi iz_1}q^j}
\\[2mm]
&=&
e^{-2\pi imt} q^{-mp^2}
\sum_{j \in \zzz}
\frac{e^{2\pi imj(z_1+z_2-p\tau)+2\pi is(z_1-p\tau)} q^{m(j+p)^2+s(j+p)}
}{1-e^{2\pi i(z_1-p\tau)}q^{j+p}}
\\[2mm]
&=&
e^{-2\pi imt} e^{-2\pi imp(z_1+z_2)}
\sum_{j \in \zzz}
\frac{e^{2\pi im(j+p)(z_1+z_2-p\tau)+2\pi is(z_1-p\tau)} q^{m(j+p)^2+s(j+p)}
}{1-e^{2\pi i(z_1-p\tau)}q^{j+p}}
\end{eqnarray*}}
Putting $j+p=:j'$, we obtain 
$$
\text{RHS of the above} \, = \, e^{-2\pi imp(z_1+z_2)}
\Phi^{[m,s]}_1(\tau, \, z_1-p\tau, \, z_2, \, t) \, ,
$$
proving 1) for $i=1$. The case $i=2$ follows from 1) and Lemma 2.2 in \cite{W2022}.

\medskip

\noindent
2) \,\ $\Phi^{[m,s]}_i(\tau, \, z_1, \, z_2+a\tau, \, t) 
= 
e^{-2\pi ima(z_1+z_2)} \, 
\underbrace{\Phi^{[m,s]}_i(\tau, \, z_1-a\tau, \, z_2, \, t)}_{
\substack{|| \\[-1mm] {\displaystyle \hspace{-5mm}
\Phi^{[m,s]}_j(\tau, \, -z_2, \, -z_1+a\tau, \, t) 
}}} $

\vspace{-2mm}

\noindent
by Lemma 2.2 in \cite{W2022}, priving 2).
\end{proof}

\begin{lemma} 
\label{n3-2:lemma:2022-324b}
Let $m \in \frac12 \nnn$, $s \in \frac12 \zzz$ and 
$p \in \zzz$. Then 
\begin{enumerate}
\item[{\rm 1)}] \,\ $\Phi^{[m,s]}_i(\tau, z_1, z_2+p\tau, t) 
\, = \, 
e^{-2\pi impz_1} \Phi^{[m,s+mp]}_i(\tau, z_1, z_2, t)$ 
\,\ for $i=1,2$.

\item[{\rm 2)}] \,\ $\Phi^{[m,s]}(\tau, z_1, z_2+p\tau, t) 
\, = \, 
e^{-2\pi impz_1} \Phi^{[m,s+mp]}(\tau, z_1, z_2, t)$ 
\end{enumerate}
\end{lemma}

\begin{proof} 1) \,\ First consider the case $i=1$;
{\allowdisplaybreaks
\begin{eqnarray*}
& &
\Phi^{[m,s]}_1(\tau, z_1, z_2+p\tau,t) 
\, = \, 
e^{-2\pi imt} \sum_{j \in \zzz}
\frac{e^{2\pi imj(z_1+z_2+p\tau)+2\pi isz_1} 
q^{mj^2+sj}}{1-e^{2\pi iz_1}q^j}
\\[2mm]
&=&
e^{-2\pi imt} e^{-2\pi impz_1}
\sum_{j \in \zzz}
\dfrac{e^{2\pi imj(z_1+z_2)+2\pi i(s+mp)z_1} 
q^{mj^2+(s+mp)j}}{1-e^{2\pi iz_1}q^j}
\\[2mm]
&=&
e^{-2\pi impz_1} \, \Phi^{[m,s+mp]}_1(\tau, z_1, z_2,t) 
\end{eqnarray*}}
proving 1) for $i=1$. 

\medskip

Next we compute the case $i=2$ by using Lemma 2.2 in \cite{W2022} 
and Lemma \ref{n3-2:lemma:2022-313b} and the formula $1)_{i=1}$:
{\allowdisplaybreaks
\begin{eqnarray*}
& & \hspace{-10mm}
\Phi^{[m,s]}_2(\tau, z_1, z_2+p\tau,t) 
\, = \, 
e^{-2\pi imp(z_1+z_2)}\Phi^{[m,s]}_2(\tau, z_1-p\tau, z_2,t)
\\[2mm]
&=&
e^{-2\pi imp(z_1+z_2)}
\Phi^{[m,s]}_1(\tau, -z_2, -z_1+p\tau,t)
\\[2mm]
&=&
e^{-2\pi imp(z_1+z_2)}e^{2\pi impz_2}
\Phi^{[m,s+mp]}_1(\tau, -z_2, -z_1,t)
\\[2mm]
&=&
e^{-2\pi impz_1} \, \Phi^{[m,s+mp]}_2(\tau, z_1, z_2,t) 
\end{eqnarray*}}
proving 1) for $i=2$. \,\ 2) follows from 1) immediately.
\end{proof}

\begin{lemma} 
\label{n3-2:lemma:2022-313a}
Let $m \in \frac12 \nnn$, \, $s \in \frac12 \zzz$ and $p \in \zzz$ 
such that $mp \in \zzz$. Then 
\begin{enumerate}
\item[{\rm 1)}] \,\ if $p \geq 0$, 
\begin{enumerate}
\item[{\rm (i)}] \,\ 
$\Phi^{[m,s]}_1(\tau, z_1, z_2+p\tau, 0) \, = \, e^{-2\pi impz_1}
\Phi^{[m,s]}_1(\tau, z_1, z_2, 0) $
$$
- \, e^{-2pi impz_1} \sum_{k=0}^{mp-1}
e^{\pi i(s+k)(z_1-z_2)} q^{-\frac{(s+k)^2}{4m}}
\theta_{s+k,m}(\tau, z_1+z_2)
$$
\item[{\rm (ii)}] \,\ 
$\Phi^{[m,s]}_2(\tau, z_1, z_2+p\tau, 0) \, = \, e^{-2\pi impz_1}
\Phi^{[m,s]}_2(\tau, z_1, z_2, 0) $
$$
- \, e^{-2pi impz_1} \sum_{k=0}^{mp-1}
e^{\pi i(s+k)(z_1-z_2)} q^{-\frac{(s+k)^2}{4m}}
\theta_{-(s+k),m}(\tau, z_1+z_2)
$$
\item[{\rm (iii)}] \,\ 
$\Phi^{[m,s]}(\tau, z_1, z_2+p\tau, 0) \, = \, e^{-2\pi impz_1}
\Phi^{[m,s]}(\tau, z_1, z_2, 0) $
$$
- \, e^{-2pi impz_1} \sum_{k=0}^{mp-1}
e^{\pi i(s+k)(z_1-z_2)} q^{-\frac{(s+k)^2}{4m}}
\big[\theta_{s+k,m}-\theta_{-(s+k),m}\big](\tau, z_1+z_2)
$$
\end{enumerate}
\item[{\rm 2)}] \,\ if $p \leq 0$, 
\begin{enumerate}
\item[{\rm (i)}] \,\ 
$\Phi^{[m,s]}_1(\tau, z_1, z_2+p\tau, 0) \, = \, e^{-2\pi impz_1}
\Phi^{[m,s]}_1(\tau, z_1, z_2, 0) $
$$
+ \, e^{-2pi impz_1} \sum_{k=mp}^{-1}
e^{\pi i(s+k)(z_1-z_2)} q^{-\frac{(s+k)^2}{4m}}
\theta_{s+k,m}(\tau, z_1+z_2)
$$
\item[{\rm (ii)}] \,\ 
$\Phi^{[m,s]}_2(\tau, z_1, z_2+p\tau, 0) \, = \, e^{-2\pi impz_1}
\Phi^{[m,s]}_2(\tau, z_1, z_2, 0) $
$$
+ \, e^{-2pi impz_1} \sum_{k=mp}^{-1}
e^{\pi i(s+k)(z_1-z_2)} q^{-\frac{(s+k)^2}{4m}}
\theta_{-(s+k),m}(\tau, z_1+z_2)
$$
\item[{\rm (iii)}] \,\ 
$\Phi^{[m,s]}(\tau, z_1, z_2+p\tau, 0) \, = \, e^{-2\pi impz_1}
\Phi^{[m,s]}(\tau, z_1, z_2, 0) $
$$
+ \, e^{-2pi impz_1} \sum_{k=mp}^{-1}
e^{\pi i(s+k)(z_1-z_2)} q^{-\frac{(s+k)^2}{4m}}
\big[\theta_{s+k,m}-\theta_{-(s+k),m}\big](\tau, z_1+z_2)
$$
\end{enumerate}
\end{enumerate}
\end{lemma}

\begin{proof} These formulas are obtained from 
Lemma \ref{n3-2:lemma:2022-324b} and Lemma \ref{n3-2:lemma:2022-312a}
as follows.

\medskip

\noindent
1) \,\ If $p \geq 0$, then 
{\allowdisplaybreaks
\begin{eqnarray*}
& & \hspace{-10mm}
\Phi^{[m,s]}_1(\tau, z_1, z_2+p\tau, 0) \, = \, e^{-2\pi impz_1}
\Phi^{[m,s+mp]}_1(\tau, z_1, z_2, 0)
\\[2mm]
&=&
e^{-2\pi impz_1} \Bigg\{
\Phi^{[m, s]}_1(\tau, z_1,z_2,0)
\\[2mm]
& &
-\sum_{k=0}^{mp-1}e^{\pi i(s+k)(z_1-z_2)} q^{-\frac{(s+k)^2}{4m}}
\theta_{s+k,m}(\tau, z_1+z_2)
\Bigg\}
\\[2mm]
& & \hspace{-10mm}
\Phi^{[m,s]}_2(\tau, z_1, z_2+p\tau, 0) \, = \, e^{-2\pi impz_1}
\Phi^{[m,s+mp]}_2(\tau, z_1, z_2, 0)
\\[2mm]
&=&
e^{-2\pi impz_1} \Bigg\{
\Phi^{[m, s]}_2(\tau, z_1,z_2,0)
\\[2mm]
& &
-\sum_{k=0}^{mp-1}e^{\pi i(s+k)(z_1-z_2)} q^{-\frac{(s+k)^2}{4m}}
\theta_{-(s+k),m}(\tau, z_1+z_2)
\Bigg\}
\end{eqnarray*}}
proving 1).

\medskip

\noindent
2) \,\ If $p \leq 0$, then 
{\allowdisplaybreaks
\begin{eqnarray*}
& & \hspace{-10mm}
\Phi^{[m,s]}_1(\tau, z_1, z_2+p\tau, 0) \, = \, e^{-2\pi impz_1}
\Phi^{[m,s+mp]}_1(\tau, z_1, z_2, 0)
\\[2mm]
&=&
e^{-2\pi impz_1} \Bigg\{
\Phi^{[m, s]}_1(\tau, z_1,z_2,0)
\\[2mm]
& &
+\sum^{-1}_{k=mp}e^{\pi i(s+k)(z_1-z_2)} q^{-\frac{(s+k)^2}{4m}}
\theta_{s+k,m}(\tau, z_1+z_2)
\Bigg\}
\\[2mm]
& & \hspace{-10mm}
\Phi^{[m,s]}_2(\tau, z_1, z_2+p\tau, 0) \, = \, e^{-2\pi impz_1}
\Phi^{[m,s+mp]}_2(\tau, z_1, z_2, 0)
\\[2mm]
&=&
e^{-2\pi impz_1} \Bigg\{
\Phi^{[m, s]}_2(\tau, z_1,z_2,0)
\\[2mm]
& &
+\sum^{-1}_{k=mp}e^{\pi i(s+k)(z_1-z_2)} q^{-\frac{(s+k)^2}{4m}}
\theta_{-(s+k),m}(\tau, z_1+z_2)
\Bigg\}
\end{eqnarray*}}
proving 2).
\end{proof}

\section{Coroot lattice of $D(2,1;a)$ and formulas for $\Phi^{[m,s]}$}

We consider the Dynkin diagram of the affine Lie superalgebra 
$\widehat{D}(2,1;a)$ 

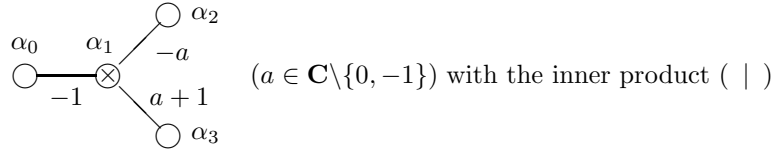
\setlength{\unitlength}{1mm}
\begin{picture}(32,13)
\put(5,1){\circle{3}}
\put(16,1){\circle{3}}
\put(24,-7){\circle{3}}
\put(24,9){\circle{3}}
\put(6.5,1){\line(1,0){8}}
\put(17.5,2.5){\line(1,1){5.5}}
\put(17.5,-0.5){\line(1,-1){5.5}}
\put(16,1){\makebox(0,0){$\times$}}
\put(5,5){\makebox(0,0){$\alpha_0$}}
\put(15,5){\makebox(0,0){$\alpha_1$}}
\put(29,9){\makebox(0,0){$\alpha_2$}}
\put(29,-7){\makebox(0,0){$\alpha_3$}}
\put(10.5,-2){\makebox(0,0){$-1$}}
\put(24.5,4){\makebox(0,0){$-a$}}
\put(25.5,-2){\makebox(0,0){$a+1$}}
\end{picture} \,\ 
$(a \in \ccc \backslash \{0, -1\}$) with the inner product $( \,\ | \,\ )$

\vspace{10mm}

\noindent
such that \,\ $\Big((\alpha_i|\alpha_j)\Big)_{i,j=0, 1,2,3} 
= \, \left(
\begin{array}{cccc}
2 & -1 & 0 & 0 \\[0.5mm]
-1 & 0 & -a & a+1 \\[0.5mm]
0 & -a & 2a & 0 \\[0.5mm]
0 & a+1 & 0 & -2(a+1)
\end{array} \right) $ .

\medskip

The coroots $\alpha_i^{\vee}= \frac{2}{|\alpha_i|^2}\alpha_i$ \,\ 
$(i=2,3)$ \,\ are given by
\begin{subequations}
\begin{equation}
\alpha_2^{\vee} = \frac{1}{a} \alpha_2, \hspace{10mm}
\alpha_3^{\vee} = \frac{-1}{a+1} \alpha_3
\label{n3-2:eqn:2022-312a1}
\end{equation}
and satisfy
\begin{equation}
|\alpha_2^{\vee}|^2 = \frac{2}{a}, \qquad 
|\alpha_3^{\vee}|^2 = \frac{-2}{a+1}, \qquad 
(\alpha_1| \alpha_2^{\vee}) = (\alpha_1| \alpha_3^{\vee}) =-1
\label{n3-2:eqn:2022-312a2}
\end{equation}
\end{subequations}

Let $\delta :=\alpha_0+2\alpha_1+\alpha_2+\alpha_3$ be the primitive 
imaginary root of $\widehat{D}(2,1,a)$, and $\Lambda_0$ be the fundamental
weight of $\widehat{D}(2,1,a)$ defined by 
$$
(\Lambda_0 | \alpha_i) \, = \, \delta_{i,0} \hspace{10mm} \text{and}
\hspace{10mm} (\Lambda_0|\Lambda_0)=0 \, .
$$
We define the coordinates on the Cartan subalgebra $\hhh$ of 
$\widehat{D}(2,1;a)$ by
\begin{subequations}
{\allowdisplaybreaks
\begin{eqnarray}
& & \hspace{-12mm}
(\tau, z_1, z_2, z_3,t)
\nonumber
\\[2mm]
& & \hspace{-15mm}
= \,\ 
2\pi i\Big\{-\tau \Lambda_0
\, - \, (z_2+z_3) \, \frac{\theta}{2}
\, + \, \frac{z_1-z_2}{2a} \, \alpha_2
\, - \, \frac{z_1-z_3}{2(a+1)} \, \alpha_3
\, + \, t \, \delta\Big\}
\nonumber
\\[2mm]
& &\hspace{-15mm}
= \,\ 
2\pi i\Big\{-\tau \Lambda_0
\, - \, (z_2+z_3) \, \frac{\theta}{2}
\, + \, \frac{z_1-z_2}{2} \, \alpha_2^{\vee}
\, + \, \frac{z_1-z_3}{2} \, \alpha_3^{\vee}
\, + \, t \, \delta\Big\}
\label{n3-2:eqn:2022-314a}
\end{eqnarray}}

\noindent
where $\theta := 2\alpha_1+\alpha_2+\alpha_3$ is the highest root.
From this definition of coordinates we have, in particular, 
the following :
\begin{equation}
(\tau, \, z_2+z_3, \, z_3-z_2, \, z_2-z_3, \, 0) 
\, = \, 2\pi i \big\{
-\tau \Lambda_0+z_2 \alpha_2^{\vee}+z_3\alpha_3^{\vee}\big\}
\label{n3-2:eqn:2022-314b}
\end{equation}
\end{subequations}

\begin{note} 
\label{n3-2:note:2022-316c}
The following formulas hold for $h=(\tau, z_1, z_2, z_3,t) \in \hhh$ :
\begin{enumerate}
\item[{\rm 1)}] \,\ $\left\{
\begin{array}{lcl}
e^{-\alpha_1(h)} &=& e^{2\pi iz_1} \\[1mm]
e^{-\alpha_2(h)} &=& e^{2\pi i(-z_1+z_2)} \\[1mm]
e^{-\alpha_3(h)} &=& e^{2\pi i(-z_1+z_3)} 
\end{array}\right.$

\item[{\rm 2)}] \,\ $\left\{
\begin{array}{lcl}
e^{-(\alpha_1+\alpha_2)(h)} &=& e^{2\pi iz_2} \\[1mm]
e^{-(\alpha_1+\alpha_3)(h)} &=& e^{2\pi iz_3} \\[1mm]
e^{-(\alpha_1+\alpha_2+\alpha_3)(h)} &=& e^{2\pi i(-z_1+z_2+z_3)} 
\end{array}\right.$

\item[{\rm 3)}] \,\ $\left\{
\begin{array}{lcl}
^{\alpha_2^{\vee}(h)} &=& e^{2\pi i \frac{1}{a}(z_1-z_2)} \\[1mm]
e^{\alpha_3^{\vee}(h)} &=& e^{2\pi i \frac{1}{a+1}(-z_1+z_3)} 
\end{array}\right.$
\end{enumerate}
\end{note}

\medskip

In this section, henceforward, we consider $\widehat{D}(2,1;a)$ 
such that 

\begin{equation}
a := \frac{-m}{m+1} \hspace{10mm} \text{where} \hspace{10mm}
m \in \tfrac12 \nnn
\label{n3-2:eqn:2022-314d}
\end{equation}
Then we have:

\begin{note}
\label{n3-2:note:2022-316a}
For $\widehat{D}(2,1;a)$ with {\rm \eqref{n3-2:eqn:2022-314d}}, 
the following formulas holds:
\begin{enumerate}
\item[{\rm 1)}] \,\ $\left\{
\begin{array}{lcl}
e^{a\alpha_2^{\vee}(h)} &=& e^{2\pi i(z_1-z_2)} \\[2mm] 
e^{a\alpha_3^{\vee}(h)} &=& e^{2\pi im(z_1-z_3)}
\end{array}\right. \hspace{3mm} \text{for} \,\ 
h=(\tau, z_1,z_2, z_3,t) \in \hhh $

\item[{\rm 2)}] \,\ $\alpha_2^{\vee} = - \frac{m+1}{m}\alpha_2 \, , 
\hspace{10mm}
\alpha_3^{\vee} = (m+1)\alpha_3$

\item[{\rm 3)}]
\begin{enumerate}
\item[{\rm (i)}] \,\ $|\alpha_2^{\vee}|^2 = - \frac{2(m+1)}{m} \, , 
\hspace{10mm}
|\alpha_3^{\vee}|^2 = -2(m+1)$

\vspace{1mm}

\item[{\rm (ii)}] \,\ $|j\alpha_2^{\vee}+k\alpha_3^{\vee}|^2 
= - \frac{2(m+1)}{m}(j^2+mk^2)$
\end{enumerate}
\end{enumerate}
\end{note}

\medskip

For $\alpha \in \hhh$ we define the linear automorphism 
$t_{\alpha}$ (cf. \cite{K1}) of $\hhh$ by
$$
t_{\alpha}(\lambda) := \lambda +(\lambda|\delta)\alpha
- \left\{\frac{(\alpha|\alpha)}{2}(\lambda|\delta)+(\lambda|\alpha)\right\} 
\, \delta \, .
$$
Then, by \eqref{n3-2:eqn:2022-312a2} and Note \ref{n3-2:note:2022-316a}, 
the action of $t_{j\alpha_2^{\vee}+k\alpha_3^{\vee}}$ is as follows:

\begin{note} \,\ 
\label{n3-2:note:2022-316b}
\begin{enumerate}
\item[{\rm 1)}] \,\ $t_{j\alpha_2^{\vee}+k\alpha_3^{\vee}}(\Lambda_0)
\, = \, \Lambda_0+j\alpha_2^{\vee}+k\alpha_3^{\vee}+ 
\frac{m+1}{m}(j^2+mk^2)\delta$

\item[{\rm 2)}] \,\ $t_{j\alpha_2^{\vee}+k\alpha_3^{\vee}}
(\alpha_1) \, = \, \alpha_1+(j+k)\delta$
\end{enumerate}
\end{note}

\medskip

For $\alpha \in \hhh$ such that $|\alpha|^2 \ne 0$, let 
$r_{\alpha}$ be the reflection on $\hhh$ with respect to $\alpha$, 
namely
$$
r_{\alpha}(\lambda) \, := \, \lambda - (\lambda|\alpha^{\vee}) \alpha 
\hspace{10mm} \text{where} \quad 
\alpha^{\vee} := \frac{2\alpha}{|\alpha|^2} \, ,
$$
We put $r_i := r_{\alpha_i} \,\ (i=2,3)$ and let $\overline{W} 
:= \langle r_2, \, r_3\rangle$ be the subgroup of 
$GL(\hhh)$ genarated by $r_i \,\ (i=2,3)$, and put 
$\varepsilon(w) := {\rm det}_{\hhh}(w)$ for $w \in \overline{W}$.
Under this setting, we compute 
{\allowdisplaybreaks
\begin{eqnarray*}
F^{(-)}_{a\Lambda_0} &:=& \sum_{j,k \in \zzz}
t_{j\alpha_2^{\vee}+k\alpha_3^{\vee}}\left(
\frac{e^{a\Lambda_0}}{1-e^{-\alpha_1}}\right) \, , \\[2mm]
A^{(-)}_{a\Lambda_0} &:=& \sum_{w \in \overline{W}}
\varepsilon(w) \, w \big(F^{(-)}_{a\Lambda_0}\big) \, .
\end{eqnarray*}}
First, by Note \ref{n3-2:note:2022-316b}, $F^{(-)}_{a\Lambda_0}$ is 
written as follows :
$$
F^{(-)}_{a\Lambda_0} = e^{a\Lambda_0}
\sum_{j,k \in \zzz}\frac{
e^{a(j\alpha_2^{\vee}+k\alpha_3^{\vee})} \, q^{j^2+mk^2}
}{1-e^{-\alpha_1} \, q^{j+k}}
$$
where $q:=e^{-\delta}$. Then, applying the elements in $\overline{W}$
to $F^{(-)}_{a\Lambda_0}$, we have
{\allowdisplaybreaks
\begin{eqnarray*}
\lefteqn{
A^{(-)}_{a\Lambda_0} \,\ = \,\ \big\{
1-r_2-r_3+r_2r_3\big\}(F^{(-)}_{a\Lambda_0})
}
\\[2mm]
&=& e^{a\Lambda_0} \Bigg\{
\sum_{j,k \in \zzz}\frac{
e^{a(j\alpha_2^{\vee}+k\alpha_3^{\vee})} \, q^{j^2+mk^2}
}{1-e^{-\alpha_1} \, q^{j+k}}
\, -
\sum_{j,k \in \zzz}\frac{
e^{a(-j\alpha_2^{\vee}+k\alpha_3^{\vee})} \, q^{j^2+mk^2}
}{1-e^{-(\alpha_1+\alpha_2)} \, q^{j+k}}
\\[2mm]
& &-
\sum_{j,k \in \zzz}\frac{
e^{a(j\alpha_2^{\vee}-k\alpha_3^{\vee})} \, q^{j^2+mk^2}
}{1-e^{-(\alpha_1+\alpha_3)} \, q^{j+k}}
\, -
\sum_{j,k \in \zzz}\frac{
e^{-a(j\alpha_2^{\vee}+k\alpha_3^{\vee})} \, q^{j^2+mk^2}
}{1-e^{-(\alpha_1+\alpha_2+\alpha_3)} \, q^{j+k}}
\Bigg\}
\end{eqnarray*}}
We write this formula by using the coordinates. 
By Note \ref{n3-2:note:2022-316c} and Note \ref{n3-2:note:2022-316a},
this formula is written for $h =(\tau, z_1, z_2, z_3, 0)$ 
as follows: 
{\allowdisplaybreaks
\begin{eqnarray}
A^{(-)}_{a\Lambda_0}(h) &=&
\sum_{j, \, k \, \in \, \zzz}
\frac{e^{2\pi ij(z_1-z_2)+2\pi ikm(z_1-z_3)} \, q^{j^2+mk^2}
}{1-e^{2\pi iz_1} \, q^{j+k}}
\nonumber
\\[2mm]
& &-
\sum_{j, \, k \, \in \, \zzz}
\frac{e^{-2\pi ij(z_1-z_2)+2\pi ikm(z_1-z_3)} \, q^{j^2+mk^2}
}{1-e^{2\pi iz_2} \, q^{j+k}}
\nonumber
\\[2mm]
& &-
\sum_{j, \, k \, \in \, \zzz}
\frac{e^{2\pi ij(z_1-z_2)-2\pi ikm(z_1-z_3)} \, q^{j^2+mk^2}
}{1-e^{2\pi iz_3} \, q^{j+k}}
\nonumber
\\[2mm]
& &+
\sum_{j, \, k \, \in \, \zzz}
\frac{e^{-2\pi ij(z_1-z_2)-2\pi ikm(z_1-z_3)} \, q^{j^2+mk^2}
}{1-e^{2\pi i(-z_1+z_2+z_3)} \, q^{j+k}}
\label{n3-2:eqn:2022-316a}
\end{eqnarray}}
We compute the RHS of this equation \eqref{n3-2:eqn:2022-316a}
in 2 ways. First, putting $k=r-j$, the RHS of 
\eqref{n3-2:eqn:2022-316a} becomes as follows:
\begin{subequations}
{\allowdisplaybreaks
\begin{eqnarray}
& & \hspace{-7mm}
\text{RHS of \eqref{n3-2:eqn:2022-316a}}
\,\ = \,\ 
\sum_{j, \, r \, \in \, \zzz}
\frac{e^{2\pi ij(z_1-z_2)+2\pi i(r-j)m(z_1-z_3)} \, q^{j^2+m(r-j)^2}
}{1-e^{2\pi iz_1} \, q^{r}}
\nonumber
\\[2mm]
& &-
\sum_{j, \, r \, \in \, \zzz}
\frac{e^{-2\pi ij(z_1-z_2)+2\pi i(r-j)m(z_1-z_3)} \, q^{j^2+m(r-j)^2}
}{1-e^{2\pi iz_2} \, q^{r}}
\nonumber
\\[2mm]
& &-
\sum_{j, \, r \, \in \, \zzz}
\frac{e^{2\pi ij(z_1-z_2)-2\pi i(r-j)m(z_1-z_3)} \, q^{j^2+m(r-j)^2}
}{1-e^{2\pi iz_3} \, q^{r}}
\nonumber
\\[2mm]
& &+
\sum_{j, \, r \, \in \, \zzz}
\frac{e^{-2\pi ij(z_1-z_2)-2\pi ikm(z_1-z_3)} \, q^{j^2+m(r-j)^2}
}{1-e^{2\pi i(-z_1+z_2+z_3)} \, q^{r}}
\nonumber
\\[2mm]
&=&
\sum_{j \in \zzz}q^{(m+1)j^2}
e^{2\pi ij(z_1-z_2)-2\pi ijm(z_1-z_3)}
\sum_{r \in \zzz}\frac{
e^{2\pi imr(z_1-z_3-2jr)}q^{mr^2}}{1-e^{2\pi iz_1}q^r}
\nonumber
\\[2mm]
& &- \, 
\sum_{j \in \zzz}q^{(m+1)j^2}
e^{-2\pi ij(z_1-z_2)-2\pi ijm(z_1-z_3)}
\sum_{r \in \zzz}\frac{
e^{2\pi imr(z_1-z_3-2jr)}q^{mr^2}}{1-e^{2\pi iz_2}q^r}
\nonumber
\\[2mm]
& &- \, 
\sum_{j \in \zzz}q^{(m+1)j^2}
e^{2\pi ij(z_1-z_2)+2\pi ijm(z_1-z_3)}
\sum_{r \in \zzz}\frac{
e^{2\pi imr(-z_1+z_3-2jr)}q^{mr^2}}{1-e^{2\pi iz_3}q^r}
\nonumber
\\[2mm]
& &+ \, 
\sum_{j \in \zzz}q^{(m+1)j^2}
e^{-2\pi ij(z_1-z_2)+2\pi ijm(z_1-z_3)}
\sum_{r \in \zzz}\frac{
e^{2\pi imr(-z_1+z_3-2jr)}q^{mr^2}}{1-e^{2\pi i(-z_1+z_2+z_3)}q^r}
\nonumber
\\[2mm]
&=&
\sum_{j \in \zzz}q^{(m+1)j^2} \, 
e^{2\pi ij(z_1-z_2)-2\pi ijm(z_1-z_3)} \, 
\Phi^{[m,0]}_1(\tau, \, z_1, \, -z_3-2j\tau, \, 0)
\nonumber
\\[2mm]
& &- 
\sum_{j \in \zzz}q^{(m+1)j^2} 
e^{-2\pi ij(z_1-z_2)-2\pi ijm(z_1-z_3)} 
\Phi^{[m,0]}_1(\tau, z_2, z_1-z_2-z_3-2j\tau, 0)
\nonumber
\\[2mm]
& &- 
\sum_{j \in \zzz}q^{(m+1)j^2} \, 
e^{2\pi ij(z_1-z_2)+2\pi ijm(z_1-z_3)} \, 
\Phi^{[m,0]}_1(\tau, \, z_3, \, -z_1-2j\tau, \, 0)
\nonumber
\\[2mm]
& &+
\sum_{j \in \zzz}q^{(m+1)j^2} 
e^{-2\pi ij(z_1-z_2)+2\pi ijm(z_1-z_3)} 
\Phi^{[m,0]}_1(\tau, -z_1+z_2+z_3, -z_3-2j\tau, 0)
\nonumber
\\[2mm]
&=&
\sum_{j \in \zzz}q^{(m+1)j^2} \, 
e^{2\pi ij(z_1-z_2)-2\pi ijm(z_1-z_3)} \, 
\Phi^{[m,0]}_1(\tau, \, z_1, \, -z_3-2j\tau, \, 0)
\nonumber
\\[2mm]
& &- 
\sum_{j \in \zzz}q^{(m+1)j^2} 
e^{-2\pi ij(z_1-z_2)-2\pi ijm(z_1-z_3)} 
\Phi^{[m,0]}_1(\tau, z_2, z_1-z_2-z_3-2j\tau, 0)
\nonumber
\\[2mm]
& &-
\sum_{j \in \zzz}q^{(m+1)j^2} \, 
e^{2\pi ij(z_1-z_2)-2\pi ijm(z_1-z_3)} \, 
\Phi^{[m,0]}_2(\tau, \, z_1, \, -z_3-2j\tau, \, 0)
\nonumber
\\[2mm]
& &+ 
\sum_{j \in \zzz}q^{(m+1)j^2} 
e^{-2\pi ij(z_1-z_2)-2\pi ijm(z_1-z_3)} 
\Phi^{[m,0]}_2(\tau, z_2, z_1-z_2-z_3-2j\tau, 0)
\nonumber
\\[2mm]
&=&
\sum_{j \in \zzz}q^{(m+1)j^2} \, 
e^{2\pi ij(z_1-z_2)-2\pi ijm(z_1-z_3)} \, 
\Phi^{[m,0]}(\tau, \, z_1, \, -z_3-2j\tau, \, 0)
\nonumber
\\[2mm]
& &- 
\sum_{j \in \zzz}q^{(m+1)j^2} 
e^{-2\pi ij(z_1-z_2)-2\pi ijm(z_1-z_3)} 
\Phi^{[m,0]}(\tau, z_2, z_1-z_2-z_3-2j\tau, 0)
\nonumber
\\[-2mm]
& &
\label{n3-2:eqn:2022-316b}
\end{eqnarray}}

\vspace{-3mm}

\noindent
by using Lemma \ref{n3-2:lemma:2022-313b}. Next we compute the the RHS 
of \eqref{n3-2:eqn:2022-316a} by putting $j=r-k$ as follows:
{\allowdisplaybreaks
\begin{eqnarray}
& & \hspace{-7mm}
\text{RHS of \eqref{n3-2:eqn:2022-316a}}
\,\ = \,\ 
\sum_{k, \, r \, \in \, \zzz}
\frac{e^{2\pi i(r-k)(z_1-z_2)+2\pi ikm(z_1-z_3)} \, q^{(r-k)^2+mk^2}
}{1-e^{2\pi iz_1} \, q^{r}}
\nonumber
\\[2mm]
& &-
\sum_{k, \, r \, \in \, \zzz}
\frac{e^{-2\pi i(r-k)(z_1-z_2)+2\pi ikm(z_1-z_3)} \, q^{(r-k)^2+mk^2}
}{1-e^{2\pi iz_2} \, q^{r}}
\nonumber
\\[2mm]
& &-
\sum_{k, \, r \, \in \, \zzz}
\frac{e^{2\pi i(r-k)(z_1-z_2)-2\pi ikm(z_1-z_3)} \, q^{(r-k)^2+mk^2}
}{1-e^{2\pi iz_3} \, q^{r}}
\nonumber
\\[2mm]
& &+
\sum_{k, \, r \, \in \, \zzz}
\frac{e^{-2\pi i(r-k)(z_1-z_2)-2\pi ikm(z_1-z_3)} \, q^{(r-k)^2+mk^2}
}{1-e^{2\pi i(-z_1+z_2+z_3)} \, q^{r}}
\nonumber
\\[2mm]
&=&
\sum_{k \in \zzz} q^{(m+1)k^2} 
e^{-2\pi ik(z_1-z_2)+2\pi ikm(z_1-z_3)} \, 
\sum_{r \in \zzz} \frac{e^{2\pi ir(z_1-z_2-2k\tau)} \, q^{r^2}
}{1-e^{2\pi iz_1} \, q^r}
\nonumber
\\[2mm]
& & -
\sum_{k \in \zzz} q^{(m+1)k^2} 
e^{2\pi ik(z_1-z_2)+2\pi ikm(z_1-z_3)} \, 
\sum_{r \in \zzz} \frac{e^{2\pi ir(-z_1+z_2-2k\tau)} \, q^{r^2}
}{1-e^{2\pi iz_2} \, q^r}
\nonumber
\\[2mm]
& & -
\sum_{k \in \zzz} q^{(m+1)k^2} 
e^{-2\pi ik(z_1-z_2)-2\pi ikm(z_1-z_3)} \, 
\sum_{r \in \zzz} \frac{e^{2\pi ir(z_1-z_2-2k\tau)} \, q^{r^2}
}{1-e^{2\pi iz_3} \, q^r}
\nonumber
\\[2mm]
& & +
\sum_{k \in \zzz} q^{(m+1)k^2} 
e^{2\pi ik(z_1-z_2)-2\pi ikm(z_1-z_3)} \, 
\sum_{r \in \zzz} \frac{e^{2\pi ir(-z_1+z_2-2k\tau)} \, q^{r^2}
}{1-e^{2\pi i(-z_1+z_2+z_3)} \, q^r}
\nonumber
\\[2mm]
&=&
\sum_{k \in \zzz} q^{(m+1)k^2} 
e^{-2\pi ik(z_1-z_2)+2\pi ikm(z_1-z_3)} \, 
\Phi^{[1,0]}_1(\tau, \, z_1, \, -z_2-2k\tau, \, 0)
\nonumber
\\[2mm]
& &-
\sum_{k \in \zzz} q^{(m+1)k^2} 
e^{2\pi ik(z_1-z_2)+2\pi ikm(z_1-z_3)} \, 
\Phi^{[1,0]}_1(\tau, \, z_2, \, -z_1-2k\tau, \, 0)
\nonumber
\\[2mm]
& &-
\sum_{k \in \zzz} q^{(m+1)k^2} 
e^{-2\pi ik(z_1-z_2)-2\pi ikm(z_1-z_3)} \, 
\Phi^{[1,0]}_1(\tau, \, z_3, \, z_1-z_2-z_3-2k\tau, \, 0)
\nonumber
\\[2mm]
& &+
\sum_{k \in \zzz} q^{(m+1)k^2} 
e^{2\pi ik(z_1-z_2)-2\pi ikm(z_1-z_3)} \, 
\Phi^{[1,0]}_1(\tau, \, -z_1+z_2+z_3, \, -z_3-2k\tau, \, 0)
\nonumber
\\[2mm]
&=&
\sum_{k \in \zzz} q^{(m+1)k^2} 
e^{-2\pi ik(z_1-z_2)+2\pi ikm(z_1-z_3)} \, 
\Phi^{[1,0]}_1(\tau, \, z_1, \, -z_2-2k\tau, \, 0)
\nonumber
\\[2mm]
& &- 
\sum_{k \in \zzz} q^{(m+1)k^2} 
e^{-2\pi ik(z_1-z_2)+2\pi ikm(z_1-z_3)} \, 
\Phi^{[1,0]}_2(\tau, \, z_1, \, -z_2-2k\tau, \, 0)
\nonumber
\\[2mm]
& &-
\sum_{k \in \zzz} q^{(m+1)k^2} 
e^{-2\pi ik(z_1-z_2)-2\pi ikm(z_1-z_3)} \, 
\Phi^{[1,0]}_1(\tau, \, z_3, \, z_1-z_2-z_3-2k\tau, \, 0)
\nonumber
\\[2mm]
& &+\sum_{k \in \zzz} q^{(m+1)k^2} 
e^{-2\pi ik(z_1-z_2)-2\pi ikm(z_1-z_3)} \, 
\Phi^{[1,0]}_2(\tau, \, z_3, \, z_1-z_2-z_3-2k\tau, \, 0)
\nonumber
\\[2mm]
&=&
\sum_{k \in \zzz} q^{(m+1)k^2} 
e^{-2\pi ik(z_1-z_2)+2\pi ikm(z_1-z_3)} \, 
\Phi^{[1,0]}(\tau, \, z_1, \, -z_2-2k\tau, \, 0)
\nonumber
\\[2mm]
& &- 
\sum_{k \in \zzz} q^{(m+1)k^2} 
e^{-2\pi ik(z_1-z_2)-2\pi ikm(z_1-z_3)} \, 
\Phi^{[1,0]}(\tau, \, z_3, \, z_1-z_2-z_3-2k\tau, \, 0)
\nonumber
\\[-2mm]
& &
\label{n3-2:eqn:2022-316c}
\end{eqnarray}}
\end{subequations}

\vspace{-3mm}

\noindent
by using Lemma \ref{n3-2:lemma:2022-313b}. 
Then, replacing $-j$ with $j$ in \eqref{n3-2:eqn:2022-316b} and 
replacing $-k$ with $j$ in \eqref{n3-2:eqn:2022-316c}, the relation 
\lq \lq \eqref{n3-2:eqn:2022-316b} = \eqref{n3-2:eqn:2022-316c}" 
gives the following lemma:

\begin{lemma} 
\label{n3-2:lemma:2022-314a}
For $m \in \frac12 \nnn$ \, the following formula holds:
{\allowdisplaybreaks
\begin{eqnarray}
& & \,\
\sum_{j \in \zzz}q^{(m+1)j^2}
e^{-2\pi ij(z_1-z_2)+2\pi ijm(z_1-z_3)} \, 
\Phi^{[m,0]}(\tau, \, z_1, \, -z_3+2j\tau, \, 0) 
\nonumber
\\[3mm]
& & \hspace{-3mm}
- \,\ \sum_{j \in \zzz}q^{(m+1)j^2}
e^{2\pi ij(z_1-z_2)+2\pi ijm(z_1-z_3)} \, 
\Phi^{[m,0]}(\tau, \, z_2, \, z_1-z_2-z_3+2j\tau, \, 0) 
\nonumber
\\[3mm]
&=& \,\ 
\sum_{j \in \zzz}q^{(m+1)j^2}
e^{2\pi ij(z_1-z_2)-2\pi ijm(z_1-z_3)} \, 
\Phi^{[1,0]}(\tau, \, z_1, \, -z_2+2j\tau, \, 0) 
\nonumber
\\[3mm]
& & \hspace{-3mm}
- \,\ \sum_{j \in \zzz}q^{(m+1)j^2}
e^{2\pi ij(z_1-z_2)+2\pi ijm(z_1-z_3)} \, 
\Phi^{[1,0]}(\tau, \, z_3, \, z_1-z_2-z_3+2j\tau, \, 0) 
\nonumber
\\[-2mm]
& &
\label{n3-2:eqn:2022-316d} 
\end{eqnarray}}
\end{lemma}

From this lemma, we obtain the following:


\begin{prop} 
\label{n3-2:prop:2022-316a}
For $m \in \frac12 \nnn$ \, the following formula holds:

\vspace{-2mm}

\begin{subequations}
{\allowdisplaybreaks
\begin{eqnarray}
& &
\theta_{0, \, m+1}\Big(\tau, \, \frac{(z_1-z_2)+m(z_1+z_3)}{m+1}\Big)
\, \Phi^{[m,0]}(\tau, \, z_1, \, -z_3, \, 0) 
\nonumber
\\[2mm]
& & \hspace{-5mm}
- \,\ 
\theta_{0, \, m+1}\Big(\tau, \, \frac{(z_1-z_2)+m(z_1-2z_2-z_3)}{m+1}\Big)
\, \Phi^{[m,0]}(\tau, \, z_2, \, z_1-z_2-z_3, \, 0)
\nonumber
\\[2mm]
& & \hspace{-5mm}
+ \, \sum_{j \, > \, 0} \, \sum_{k=1}^{2mj} 
q^{(m+1)j^2} \, q^{-\frac{k^2}{4m}} \, P^{[m]}_{j,k}(z_1, z_2, z_3) \, 
\big[\theta_{k,m}-\theta_{-k,m}\big](\tau, z_1-z_3)
\nonumber
\\[2mm]
&=&
i \, \theta_{0, m+1}\Big(\tau, \,\ \frac{(z_1+z_2) + m(z_1-z_3)}{m+1}\Big)
\,\ 
\frac{\eta(\tau)^3 \, \vartheta_{11}(\tau, \, z_1-z_2)}{
\vartheta_{11}(\tau, z_1) \, \vartheta_{11}(\tau, z_2)}
\nonumber
\\[2mm]
& &\hspace{-5mm}
+ \,\ i \, 
\theta_{0, m+1}\Big( \tau, \,\ \frac{(z_1-z_2-2z_3) \, + \, m(z_1-z_3)}{m+1}\Big)
\,
\frac{\eta(\tau)^3 \, \vartheta_{11}(\tau, \, z_1-z_2)}{
\vartheta_{11}(\tau, z_3) \, \vartheta_{11}(\tau, \, z_1-z_2-z_3)}
\nonumber
\\[0mm]
& &
\label{n3-2:eqn:2022-316e}
\end{eqnarray}}

\vspace{-8mm}

\noindent
where

\vspace{-3mm}

{\allowdisplaybreaks
\begin{eqnarray}
& & \hspace{-10mm}
P^{[m]}_{j,k}(z_1, z_2, z_3)
\nonumber
\\[3mm]
&:=&
e^{2\pi ij(z_1-z_2)} \, e^{\pi i(2mj-k)(z_1-2z_2-z_3)}
\, + \, 
e^{-2\pi ij(z_1-z_2)} \, e^{-\pi i(2mj-k)(z_1-2z_2-z_3)}
\nonumber
\\[3mm]
& & \hspace{2mm}
- \,\ 
e^{2\pi ij(z_1-z_2)} \, e^{\pi i(2mj-k)(z_1+z_3)}
\, - \, 
e^{-2\pi ij(z_1-z_2)} \, e^{-\pi i(2mj-k)(z_1+z_3)}
\nonumber
\\[-2mm]
& &
\label{n3-2:eqn:2022-316f}
\end{eqnarray}}
\end{subequations}
and \, $\vartheta_{ab}(\tau, z)$'s \, $(a, b \in \{0,1\})$ are the Mumford's 
theta functions (\cite{Mum}).
\end{prop}

\begin{proof} In order to prove this proposition, first we compute 
{\allowdisplaybreaks
\begin{eqnarray*}
& & \hspace{-5mm}
\text{LHS of \eqref{n3-2:eqn:2022-316d}}
\\[2mm]
&=&
\underbrace{\sum_{j \in \zzz}q^{(m+1)j^2}
e^{-2\pi ij(z_1-z_2)+2\pi ijm(z_1-z_3)} \, 
\Phi^{[m,0]}(\tau, \, z_1, \, -z_3+2j\tau, \, 0)}_{\rm (I)} 
\nonumber
\\[2mm]
& & \hspace{-3mm}
\underbrace{- \sum_{j \in \zzz}q^{(m+1)j^2}
e^{2\pi ij(z_1-z_2)+2\pi ijm(z_1-z_3)} \, 
\Phi^{[m,0]}(\tau, \, z_2, \, z_1-z_2-z_3+2j\tau, \, 0)}_{\rm (II)} 
\end{eqnarray*}}
We compute (I) and (II) by using Lemma \ref{n3-2:lemma:2022-313a} as follows: 
{\allowdisplaybreaks
\begin{eqnarray*}
& & \hspace{-5mm}
{\rm (I)} \, = \, \sum_{j \geq 0}q^{(m+1)j^2}
e^{-2\pi ij(z_1-z_2)+2\pi ijm(z_1-z_3)} \, \bigg\{
e^{-4\pi imjz_1}\Phi^{[m,0]}(\tau, \, z_1, \, -z_3, \, 0)
\\[2mm]
& &
- e^{-4\pi imjz_1}\sum_{k=0}^{2mj-1}e^{\pi ik(z_1+z_3)}q^{-\frac{k^2}{4m}}
\big[\theta_{k,m}-\theta_{-k,m}\big](\tau, z_1-z_3)\bigg\}
\\[2mm]
& &\hspace{-3mm}
+ \sum_{j < 0}q^{(m+1)j^2}
e^{-2\pi ij(z_1-z_2)+2\pi ijm(z_1-z_3)} \, \bigg\{
e^{-4\pi imjz_1}\Phi^{[m,0]}(\tau, \, z_1, \, -z_3, \, 0)
\\[2mm]
& &
+ e^{-4\pi imjz_1}\sum_{k=2mj}^{-1}e^{\pi ik(z_1+z_3)}q^{-\frac{k^2}{4m}}
\big[\theta_{k,m}-\theta_{-k,m}\big](\tau, z_1-z_3)\bigg\}
\\[2mm]
&=&
\underbrace{\sum_{j \in \zzz}q^{(m+1)j^2}
e^{-2\pi ij(z_1-z_2)-2\pi ijm(z_1+z_3)} }_{\substack{|| \\[0mm] 
{\displaystyle 
\theta_{0,m+1}\Big(\tau, -\frac{z_1-z_2+m(z_1+z_3)}{m+1}\Big)
}}}
\Phi^{[m,0]}(\tau, \, z_1, \, -z_3, \, 0)
+{\rm (I)}' +{\rm (I)}''
\end{eqnarray*}}
where
{\allowdisplaybreaks
\begin{eqnarray*}
{\rm (I)}' &:=& -\sum_{j > 0}q^{(m+1)j^2}
e^{-2\pi ij(z_1-z_2)-2\pi ijm(z_1+z_3)} 
\\[2mm]
& &
\times \sum_{k=0}^{2mj-1}e^{\pi ik(z_1+z_3)}q^{-\frac{k^2}{4m}}
\big[\theta_{k,m}-\theta_{-k,m}\big](\tau, z_1-z_3)
\\[2mm]
&=&
-\sum_{j > 0}\sum_{k=0}^{2mj-1}q^{(m+1)j^2-\frac{k^2}{4m}}
e^{-2\pi ij(z_1-z_2)-\pi i(2mj-k)(z_1+z_3)} 
\\[0mm]
& & \hspace{30mm}
\times \,\ 
\big[\theta_{k,m}-\theta_{-k,m}\big](\tau, z_1-z_3)
\\[2mm]
&=&
-\sum_{j > 0}\sum_{k=1}^{2mj}q^{(m+1)j^2-\frac{k^2}{4m}}
e^{-2\pi ij(z_1-z_2)-\pi i(2mj-k)(z_1+z_3)} 
\\[0mm]
& & \hspace{30mm}
\times \,\ 
\big[\theta_{k,m}-\theta_{-k,m}\big](\tau, z_1-z_3)
\\[2mm]
{\rm (I)}'' &:=& \sum_{j < 0}q^{(m+1)j^2}
e^{-2\pi ij(z_1-z_2)-2\pi ijm(z_1+z_3)} 
\\[2mm]
& &
\times \sum_{k=2mj}^{-1}e^{\pi ik(z_1+z_3)}q^{-\frac{k^2}{4m}}
\big[\theta_{k,m}-\theta_{-k,m}\big](\tau, z_1-z_3)
\\[1mm]
&=& \quad \text{letting} \,\ (j,k) \, \rightarrow \, (-j, -k)
\\[1mm]
&=&
- \sum_{j > 0}\sum_{k=1}^{2mj}q^{(m+1)j^2-\frac{k^2}{4m}}
e^{2\pi ij(z_1-z_2)+\pi i(2mj-k)(z_1+z_3)} 
\\[0mm]
& & \hspace{30mm}
\times \,\ 
\big[\theta_{k,m}-\theta_{-k,m}\big](\tau, z_1-z_3)
\end{eqnarray*}}
Then we have
\begin{subequations}
{\allowdisplaybreaks
\begin{eqnarray}
& & \hspace{-7mm}
{\rm (I)} \, = \, \theta_{0,m+1}\Big(\tau, -\frac{z_1-z_2+m(z_1+z_3)}{m+1}\Big)
\Phi^{[m,0]}(\tau, \, z_1, \, -z_3, \, 0)
\nonumber
\\[2mm]
& &
- \sum_{j > 0} \, \sum_{k=1}^{2mj}
q^{(m+1)j^2-\frac{k^2}{4m}}
\big\{e^{-2\pi ij(z_1-z_2)-\pi i(2mj-k)(z_1+z_3)} \hspace{20mm}
\nonumber
\\[1mm]
& & \hspace{5mm}
+e^{2\pi ij(z_1-z_2)+\pi i(2mj-k)(z_1+z_3)}
\big\} 
\big[\theta_{k,m}-\theta_{-k,m}\big](\tau, z_1-z_3)
\label{n3-2:eqn:2022-317a1}
\end{eqnarray}}
\noindent
Next we compute (II):
{\allowdisplaybreaks
\begin{eqnarray*}
& & \hspace{-8mm}
{\rm (II)} \,\ = 
\\[2mm]
& & \hspace{-8mm}
- \sum_{j \geq 0}q^{(m+1)j^2}
e^{2\pi ij(z_1-z_2)+2\pi ijm(z_1-z_3)} \bigg\{
e^{-4\pi imjz_2} \Phi^{[m,0]}(\tau, z_2, z_1-z_2-z_3, 0)
\\[0mm]
& &
- \,\ e^{-4\pi imjz_2}\sum_{k=0}^{2mj-1}
e^{\pi ik(-z_1+2z_2+z_3)} q^{-\frac{k^2}{4m}}
\big[\theta_{k,m}-\theta_{-k,m}\big](\tau, z_1-z_3)\bigg\}
\\[2mm]
& &\hspace{-8mm}
- \sum_{j<0}q^{(m+1)j^2}
e^{2\pi ij(z_1-z_2)+2\pi ijm(z_1-z_3)} \bigg\{
e^{-4\pi imjz_2} \Phi^{[m,0]}(\tau, z_2, z_1-z_2-z_3, 0)
\\[0mm]
& &
+ \,\ e^{-4\pi imjz_2}\sum_{k=2mj}^{-1}
e^{\pi ik(-z_1+2z_2+z_3)} q^{-\frac{k^2}{4m}}
\big[\theta_{k,m}-\theta_{-k,m}\big](\tau, z_1-z_3)\bigg\}
\\[2mm]
& & \hspace{-8mm}
= \, - \, \underbrace{\sum_{j \in \zzz}q^{(m+1)j^2}
e^{2\pi ij(z_1-z_2)+2\pi ijm(z_1-2z_2-z_3)}}_{
\substack{|| \\[0mm] {\displaystyle 
\theta_{0,m+1}\Big(\tau, \frac{z_1-z_2+m(z_1-2z_2-z_3)}{m+1}\Big)
}}}
\Phi^{[m,0]}(\tau, z_2, z_1-z_2-z_3, 0)
\\[2mm]
& &
+ \,\ {\rm (II)}' \, + \,\ {\rm (II)}''
\end{eqnarray*}}
where
{\allowdisplaybreaks
\begin{eqnarray*}
{\rm (II)}' &:=& 
\sum_{j \geq 0}q^{(m+1)j^2}
e^{2\pi ij(z_1-z_2)+2\pi ijm(z_1-2z_2-z_3)} 
\\[2mm]
& &
\times \, \sum_{k=0}^{2mj-1}
e^{\pi ik(-z_1+2z_2+z_3)} q^{-\frac{k^2}{4m}}
\big[\theta_{k,m}-\theta_{-k,m}\big](\tau, z_1-z_3)
\\[2mm]
&=&
\sum_{j > 0}\sum_{k=0}^{2mj-1}
q^{(m+1)j^2-\frac{k^2}{4m}}
e^{2\pi ij(z_1-z_2)+\pi i(2jm-k)(z_1-2z_2-z_3)} 
\\[0mm]
& & \hspace{20mm}
\times \, 
\big[\theta_{k,m}-\theta_{-k,m}\big](\tau, z_1-z_3)
\\[2mm]
&=&
\sum_{j > 0} \, \sum_{k=1}^{2mj}
q^{(m+1)j^2-\frac{k^2}{4m}}
e^{2\pi ij(z_1-z_2)+\pi i(2jm-k)(z_1-2z_2-z_3)} 
\\[-1mm]
& & \hspace{20mm}
\times \, 
\big[\theta_{k,m}-\theta_{-k,m}\big](\tau, z_1-z_3)
\\[3mm]
{\rm (II)}'' &:=& 
-\sum_{j < 0}q^{(m+1)j^2}
e^{2\pi ij(z_1-z_2)+2\pi ijm(z_1-2z_2-z_3)} 
\\[2mm]
& &
\times \, \sum_{k=-2mj}^{-1}
e^{\pi ik(-z_1+2z_2+z_3)} q^{-\frac{k^2}{4m}}
\big[\theta_{k,m}-\theta_{-k,m}\big](\tau, z_1-z_3)
\\[2mm]
&=& \quad \text{letting} \,\ (j,k) \, \rightarrow \, (-j,-k)
\\[2mm]
&=&
\sum_{j > 0} \, \sum_{k=1}^{2mj} \, 
q^{(m+1)j^2-\frac{k^2}{4m}}
e^{-2\pi ij(z_1-z_2)-\pi i(2jm-k)(z_1-2z_2-z_3)} 
\\[-1mm]
& & \hspace{20mm}
\times \,\ 
\big[\theta_{k,m}-\theta_{-k,m}\big](\tau, z_1-z_3)
\end{eqnarray*}}
Then we have
{\allowdisplaybreaks
\begin{eqnarray}
& & \hspace{-7mm}
{\rm (II)} \, = \, 
\theta_{0,m+1}\Big(\tau, \frac{z_1-z_2+m(z_1-2z_2-z_3)}{m+1}\Big)
\Phi^{[m,0]}(\tau, \, z_2, \, z_1-z_2-z_3, \, 0)
\nonumber
\\[2mm]
& &
+ \sum_{j > 0} \, \sum_{k=1}^{2mj}
q^{(m+1)j^2-\frac{k^2}{4m}}
\big\{e^{2\pi ij(z_1-z_2)+\pi i(2mj-k)(z_1-2z_2-z_3)} \hspace{20mm}
\nonumber
\\[1mm]
& & 
+ \, e^{-2\pi ij(z_1-z_2)-\pi i(2mj-k)(z_1-2z_2-z_3)}
\big\} 
\big[\theta_{k,m}-\theta_{-k,m}\big](\tau, z_1-z_3)
\label{n3-2:eqn:2022-317a2}
\end{eqnarray}}
\end{subequations}
Then, by \eqref{n3-2:eqn:2022-317a1} and \eqref{n3-2:eqn:2022-317a2}
and \eqref{n3-2:eqn:2022-316f}, the LHS of \eqref{n3-2:eqn:2022-316d}
is written as follows:
\begin{subequations}
{\allowdisplaybreaks
\begin{eqnarray}
& & \hspace{-7mm}
\text{LHS of \eqref{n3-2:eqn:2022-316d}} \,\ = \,\ 
\theta_{0, \, m+1}\Big(\tau, \, \frac{(z_1-z_2)+m(z_1+z_3)}{m+1}\Big)
\, \Phi^{[m,0]}(\tau, \, z_1, \, -z_3, \, 0) 
\nonumber
\\[2mm]
& & \hspace{-5mm}
- \,\ 
\theta_{0, \, m+1}\Big(\tau, \, \frac{(z_1-z_2)+m(z_1-2z_2-z_3)}{m+1}\Big)
\, \Phi^{[m,0]}(\tau, \, z_2, \, z_1-z_2-z_3, \, 0)
\nonumber
\\[2mm]
& & \hspace{-5mm}
+ \, \sum_{j \, > \, 0} \, \sum_{k=1}^{2mj} 
q^{(m+1)j^2-\frac{k^2}{4m}} \, P^{[m]}_{j,k}(z_1, z_2, z_3) \, 
\big[\theta_{k,m}-\theta_{-k,m}\big](\tau, z_1-z_3)
\label{n3-2:eqn:2022-317b1}
\end{eqnarray}}

\noindent
To compute the RHS of \eqref{n3-2:eqn:2022-316d}, 
we use Lemma 2.7 in \cite{W2022}: 
$$
\Phi^{[1,0]}(\tau, z_1, z_2,0) \,\ = \,\ - \, i \, 
\frac{\eta(\tau)^3 \, \vartheta_{11}(\tau, z_1+z_2)}{
\vartheta_{11}(\tau, z_1) \, \vartheta_{11}(\tau, z_2)}
$$
and the formula for the Mumford's $\vartheta$-function:
$$
\vartheta_{11}(\tau, z+n\tau) \, = \, 
(-1)^n q^{-\frac{n^2}{2}}e^{-2pi inz}\vartheta_{11}(\tau, z)
\hspace{10mm} \text{for} \quad n \in \zzz
$$
Using these formulas, we have 
{\allowdisplaybreaks
\begin{eqnarray*}
& & \hspace{-10mm}
\Phi^{[1,0]}(\tau, z_1, -z_2+2j\tau,0) 
\, = \, - \, i \, 
\frac{\eta(\tau)^3 \, \vartheta_{11}(\tau, z_1-z_2+2j\tau)
}{
\vartheta_{11}(\tau, z_1) \, \vartheta_{11}(\tau, -z_2+2j\tau)}
\\[2mm]
&=&
i \, e^{-4\pi ijz_1} \, 
\frac{\eta(\tau)^3 \, \vartheta_{11}(\tau, z_1-z_2)}{
\vartheta_{11}(\tau, z_1) \, \vartheta_{11}(\tau, z_2)}
\\[2mm]
& & \hspace{-10mm}
\Phi^{[1,0]}(\tau, z_3, z_1-z_2-z_3+2j\tau,0) 
\, = \, - \, i \, 
\frac{\eta(\tau)^3 \, \vartheta_{11}(\tau, z_1-z_2+2j\tau)
}{
\vartheta_{11}(\tau, z_3) \, \vartheta_{11}(\tau, z_1-z_2-z_3+2j\tau)}
\\[2mm]
&=&
- \, i \, e^{-4\pi ijz_3} \, 
\frac{\eta(\tau)^3 \, \vartheta_{11}(\tau, z_1-z_2)}{
\vartheta_{11}(\tau, z_3) \, \vartheta_{11}(\tau, z_1-z_2-z_3)}
\end{eqnarray*}}
so 
{\allowdisplaybreaks
\begin{eqnarray}
& & \hspace{-7mm}
\text{RHS of \eqref{n3-2:eqn:2022-316d}} 
\nonumber
\\[2mm]
&=&
\sum_{j \in \zzz}q^{(m+1)j^2}
e^{2\pi ij(z_1-z_2)-2\pi imj(z_1-z_3)} \times 
i \, e^{-4\pi ijz_1} \, 
\frac{\eta(\tau)^3 \, \vartheta_{11}(\tau, z_1-z_2)}{
\vartheta_{11}(\tau, z_1) \, \vartheta_{11}(\tau, z_2)}
\nonumber
\\[2mm]
& &+ \,\ \sum_{j \in \zzz}q^{(m+1)j^2}
e^{2\pi ij(z_1-z_2)+2\pi imj(z_1-z_3)} 
\nonumber
\\[-2mm]
& & \hspace{20mm}
\times \,\ i \, e^{-4\pi ijz_3} \, 
\frac{\eta(\tau)^3 \, \vartheta_{11}(\tau, z_1-z_2)}{
\vartheta_{11}(\tau, z_3) \, \vartheta_{11}(\tau, z_1-z_2-z_3)}
\nonumber
\\[2mm]
&=&
i \, 
\underbrace{\sum_{j \in \zzz}q^{(m+1)j^2}
e^{-2\pi ij(z_1+z_2)-2\pi imj(z_1-z_3)} }_{\substack{|| \\[0mm] 
{\displaystyle 
\theta_{0,m+1}\Big(\tau, - \frac{z_1+z_2+m(z_1-z_3)}{m+1}\Big)
}}}
\frac{\eta(\tau)^3 \, \vartheta_{11}(\tau, z_1-z_2)}{
\vartheta_{11}(\tau, z_1) \, \vartheta_{11}(\tau, z_2)}
\nonumber
\\[2mm]
& & \hspace{-5mm}
+ \, i \, 
\underbrace{\sum_{j \in \zzz}q^{(m+1)j^2}
e^{2\pi ij(z_1-z_2-2z_3)+2\pi imj(z_1-z_3)} }_{\substack{|| \\[0mm] 
{\displaystyle 
\theta_{0,m+1}\Big(\tau, \frac{z_1-z_2-2z_3+m(z_1-z_3)}{m+1}\Big)
}}}
\frac{\eta(\tau)^3 \, \vartheta_{11}(\tau, z_1-z_2)}{
\vartheta_{11}(\tau, z_3) \, \vartheta_{11}(\tau, z_1-z_2-z_3)}
\nonumber
\\[2mm]
&=&
i \, \theta_{0,m+1}\Big(\tau, \frac{z_1+z_2+m(z_1-z_3)}{m+1}\Big)
\frac{\eta(\tau)^3 \, \vartheta_{11}(\tau, z_1-z_2)}{
\vartheta_{11}(\tau, z_1) \, \vartheta_{11}(\tau, z_2)}
\nonumber
\\[2mm]
& & \hspace{-3mm}
+ \, i \, 
\theta_{0,m+1}\Big(\tau, \frac{z_1-z_2-2z_3+m(z_1-z_3)}{m+1}\Big)
\frac{\eta(\tau)^3 \, \vartheta_{11}(\tau, z_1-z_2)}{
\vartheta_{11}(\tau, z_3) \, \vartheta_{11}(\tau, z_1-z_2-z_3)}
\nonumber
\\[-0mm]
& &
\label{n3-2:eqn:2022-317b2}
\end{eqnarray}}
\end{subequations}
Then by \eqref{n3-2:eqn:2022-317b1} and \eqref{n3-2:eqn:2022-317b2}
we obtain \eqref{n3-2:eqn:2022-316e}, which completes proof of 
Proposition \ref{n3-2:prop:2022-316a}.
\end{proof}

\begin{prop} \quad 
\label{n3-2:prop:2022-314a}
For $m \in \nnn$, the following formula holds:
{\allowdisplaybreaks
\begin{eqnarray}
& &
\theta_{0, m+1}\Big(\tau, \frac{-\frac12+m(z_1+z_3)}{m+1} \Big) 
\nonumber
\\[2mm]
& &
\times \, \bigg\{
\Phi^{[m,0]}(\tau, \, z_1, \, -z_3, \, 0) 
-
\Phi^{[m,0]}\Big(\tau, \, z_1+\frac12, \, -z_3-\frac12, \, 0\Big) \bigg\}
\nonumber
\\[2mm]
&=&
2i \, \eta(\tau)^2 \eta(2\tau)^2 \Bigg\{- 
\frac{\displaystyle 
\theta_{0, m+1}\Big(\tau, \, \frac{\frac12+z_1+z_3}{m+1}+ z_1-z_3\Big)
}{\vartheta_{11}(\tau, z_1) \, \vartheta_{10}(\tau, z_1)} 
\nonumber
\\[2mm]
& & \hspace{30mm}
+ \, 
\frac{\displaystyle 
\theta_{0, m+1}\Big(\tau, \, \frac{\frac12+z_1+z_3}{m+1}- z_1+z_3\Big)
}{\vartheta_{11}(\tau, z_3) \, \vartheta_{10}(\tau, z_3)} \Bigg\}
\nonumber
\\[2mm]
& & \hspace{-3mm}
+ \, 2 \, \sum_{j=1}^{\infty} \sum_{r=1}^j 
\sum_{\substack{k=1 \\[1mm] k \, : \, {\rm odd}}}^{m-1}
(-1)^j q^{(m+1)j^2-\frac{1}{4m}(2m(j-r)+k)^2}
\nonumber
\\[2mm]
& & 
\times \, 
\big\{e^{\pi i(2mr-k)(z_1+z_3)}+e^{-\pi i(2mr-k)(z_1+z_3)}\big\}
\big[\theta_{k,m}-\theta_{-k,m}](\tau, z_1-z_3)
\nonumber
\\[2mm]
& & \hspace{-3mm}
- \, 2 \, \sum_{j=1}^{\infty} \sum_{r=0}^{j-1} 
\sum_{\substack{k=1 \\[1mm] k \, : \, {\rm odd}}}^{m-1}
(-1)^j q^{(m+1)j^2-\frac{1}{4m}(2m(j-r)-k)^2}
\nonumber
\\[2mm]
& & 
\times 
\big\{e^{\pi i(2mr+k)(z_1+z_3)}+e^{-\pi i(2mr+k)(z_1+z_3)}\big\}
\big[\theta_{k,m}-\theta_{-k,m}](\tau, z_1-z_3)
\nonumber
\\[0mm]
& &
\label{n3-2:eqn:2022-317c}
\end{eqnarray}}
\end{prop}

\begin{proof} In order to prove this proposition, we let 
$z_2=z_1+\frac12$ in the formula \eqref{n3-2:eqn:2022-316e}.
Since 
$$\left\{
\begin{array}{ccl}
z_1-2z_2-z_3 &=&-1-(z_2+z_3) \\[2mm]
\dfrac{z_1-z_2+m(z_1-2z_2-z_3)}{m+1} &=&
-1+\dfrac{\frac12-m(z_1+z_3)}{m+1}
\end{array}\right. \, ,
$$
the 1st and 2nd terms in the LHS of \eqref{n3-2:eqn:2022-316e} 
becomes as follows:
{\allowdisplaybreaks
\begin{eqnarray*}
& & \hspace{-10mm}
{\rm (I)} \, \overset{\substack{put \\[0.7mm] }}{:=} \, 
\text{(1st term}+ \text{2nd term) in LHS of \eqref{n3-2:eqn:2022-316e}}
\big|_{z_2=z_1+\frac12}
\\[2mm]
&=&
\theta_{0,m+1}\Big(\tau, \frac{-\frac12+m(z_1+z_3)}{m+1}\Big)
\Phi^{[m,0]}(\tau, z_1, -z_3,0)
\\[2mm]
& &
- 
\underbrace{
\theta_{0,m+1}\Big(\tau, -1+\frac{\frac12-m(z_1+z_3)}{m+1}\Big)}_{
\substack{|| \\[0mm] {\displaystyle 
\theta_{0,m+1}\Big(\tau, \frac{\frac12-m(z_1+z_3)}{m+1}\Big)
}}}
\Phi^{[m,0]}\Big(\tau, z_1+\frac12, -z_3-\frac12,0\Big)
\\[2mm]
&=&
\theta_{0,m+1}\Big(\tau, \frac{-\frac12+m(z_1+z_3)}{m+1}\Big)
\\[2mm]
& &
\times \bigg\{
\Phi^{[m,0]}(\tau, z_1, -z_3,0)
-
\Phi^{[m,0]}\Big(\tau, z_1+\frac12, -z_3-\frac12,0\Big)\bigg\}
\end{eqnarray*}}

Next we compute the 3rd term in the LHS of \eqref{n3-2:eqn:2022-316e}.
Since, by easy calculation, we have 
{\allowdisplaybreaks
\begin{eqnarray*}
& &
P^{[m]}_{j,k}(z_1,z_2,z_3)\big|_{z_2=z_1+\frac12}
\\[2mm]
&=&
(-1)^j\big\{(-1)^{2mj+k}-1\big\}
\big\{e^{\pi i(2mj-k)(z_1+z_3)}+e^{-\pi i(2mj-k)(z_1+z_3)}\big\} \, ,
\end{eqnarray*}}
we see that $P^{[m]}_{j,k}(z_1,z_2,z_3)\big|_{z_2=z_1+\frac12}$ is 
$\ne 0$ only if $k$ is odd. Then we have
{\allowdisplaybreaks
\begin{eqnarray*}
& & \hspace{-10mm}
{\rm (II)} \, \overset{\substack{put \\[0.7mm]}}{:=} \, 
\text{the 3rd term in LHS of \eqref{n3-2:eqn:2022-316e}}\big|_{z_2=z_1+\frac12}
\\[0mm]
&=&
- \, 2 \, 
\sum_{j>0}\sum_{\substack{k=1 \\[1mm] k \, : \, {\rm odd}}}^{2mj}
(-1)^j q^{(m+1)j^2-\frac{k^2}{4m}}
\\[2mm]
& & 
\times \, 
\big\{e^{\pi i(2mj-k)(z_1+z_3)}+e^{-\pi i(2mj-k)(z_1+z_3)}\big\}
\big[\theta_{k,m}-\theta_{-k,m}](\tau, z_1-z_3)
\\[2mm]
&=&
2 \, 
\sum_{j>0}\sum_{\substack{k=1 \\[1mm] k \, : \, {\rm odd}}}^{2mj}
(-1)^j q^{(m+1)j^2-\frac{k^2}{4m}}
\\[2mm]
& & \hspace{-10mm}
\times \, 
\big\{e^{\pi i(2mj-k)(z_1+z_3)}+e^{-\pi i(2mj-k)(z_1+z_3)}\big\}
\big[\theta_{2mj-k,m}-\theta_{-(2mj-k),m}](\tau, z_1-z_3)
\end{eqnarray*}}
Replacing $2mj-k$ with $(2r-1)m + s$ and $2(r-1)m+s$ 
($1 \leq r \leq j$ and $1\leq s<m$), this is rewritten as 
$$
{\rm (II)} \, = \, {\rm (II)}_A+ {\rm (II)}_B
$$
\noindent
where
{\allowdisplaybreaks
\begin{eqnarray*}
& & \hspace{-8mm}
{\rm (II)}_A \,\ := \,\ {\rm (II)}_{2mj-k \, = \, (2r-1)m+s}
\\[2mm]
&=&
2 \, \sum_{j=1}^{\infty} \sum_{r=1}^j 
\sum_{\substack{s=1 \\[1mm] s-m \, : \, {\rm odd}}}^{m-1}
(-1)^j q^{(m+1)j^2-\frac{1}{4m}(2m(j-r)+m-s)^2}
\\[2mm]
& & 
\times \, 
\big\{e^{\pi i((2r-1)m+s)(z_1+z_3)}+e^{-\pi i((2r-1)m+s)(z_1+z_3)}\big\}
\\[2mm]
& &
\times \,
\big[\theta_{(2r-1)m+s,m}-\theta_{-(2r-1)m-s,m}](\tau, z_1-z_3)
\\[2mm]
&=&
- \, 2 \, \sum_{j=1}^{\infty} \sum_{r=1}^j 
\sum_{\substack{s=1 \\[1mm] s+m \, : \, {\rm odd}}}^{m-1}
(-1)^j q^{(m+1)j^2-\frac{1}{4m}(2m(j-r)+m-s)^2}
\\[2mm]
& & 
\times \, 
\big\{e^{\pi i(2mr-m+s)(z_1+z_3)}+e^{-\pi i(2mr-m+s)(z_1+z_3)}\big\}
\\[2mm]
& &
\times \,
\big[\theta_{m-s,m}-\theta_{-(m-s),m}](\tau, z_1-z_3)
\\[2mm]
&=& \hspace{5mm} \text{putting} \quad m-s=:k
\\[2mm]
&=&
- \, 2 \, \sum_{j=1}^{\infty} \sum_{r=1}^j 
\sum_{\substack{k=1 \\[1mm] k \, : \, {\rm odd}}}^{m-1}
(-1)^j q^{(m+1)j^2-\frac{1}{4m}(2m(j-r)+k)^2}
\\[2mm]
& & 
\times \, 
\big\{e^{\pi i(2mr-k)(z_1+z_3)}+e^{-\pi i(2mr-k)(z_1+z_3)}\big\}
\big[\theta_{k,m}-\theta_{-k,m}](\tau, z_1-z_3)
\\[3mm]
& & \hspace{-8mm}
{\rm (II)}_B \,\ := \,\ {\rm (II)}_{2mj-k \, = \, 2(r-1)m+s}
\\[2mm]
&=&
2 \, \sum_{j=1}^{\infty} \sum_{r=1}^j 
\sum_{\substack{s=1 \\[1mm] s+m \, : \, {\rm odd}}}^{m-1}
(-1)^j q^{(m+1)j^2-\frac{1}{4m}(2m(j-r+1)-s)^2}
\\[2mm]
& & 
\times \, 
\big\{e^{\pi i(2(r-1)m+s)(z_1+z_3)}+e^{-\pi i(2(r-1)m+s)(z_1+z_3)}\big\}
\\[2mm]
& &
\times \,
\big[\theta_{2(r-1)m+s,m}-\theta_{-2(r-1)m-s,m}](\tau, z_1-z_3)
\\[2mm]
&=&
2 \, \sum_{j=1}^{\infty} \sum_{r=1}^j 
\sum_{\substack{s=1 \\[1mm] s \, : \, {\rm odd}}}^{m-1}
(-1)^j q^{(m+1)j^2-\frac{1}{4m}(2m(j-r+1)-s)^2}
\\[2mm]
& & \hspace{-7mm}
\times 
\big\{e^{\pi i(2m(r-1)+s)(z_1+z_3)}+e^{-\pi i(2m(r-1)+s)(z_1+z_3)}\big\}
\big[\theta_{s,m}-\theta_{-s,m}](\tau, z_1-z_3)
\\[2mm]
&=& \hspace{5mm} \text{replacing} \,\ r-1 \,\ \text{with} \,\ r
\\[2mm]
&=&
2 \, \sum_{j=1}^{\infty} \sum_{r=0}^{j-1} 
\sum_{\substack{k=1 \\[1mm] k \, : \, {\rm odd}}}^{m-1}
(-1)^j q^{(m+1)j^2-\frac{1}{4m}(2m(j-r)-k)^2}
\\[2mm]
& & 
\times 
\big\{e^{\pi i(2mr+k)(z_1+z_3)}+e^{-\pi i(2mr+k)(z_1+z_3)}\big\}
\big[\theta_{k,m}-\theta_{-k,m}](\tau, z_1-z_3)
\end{eqnarray*}}

\noindent
So (II) becomes as follows:
{\allowdisplaybreaks
\begin{eqnarray*}
& &
{\rm (II)} \,\ = \,\ 
- \, 2 \, \sum_{j=1}^{\infty} \sum_{r=1}^j 
\sum_{\substack{k=1 \\[1mm] k \, : \, {\rm odd}}}^{m-1}
(-1)^j q^{(m+1)j^2-\frac{1}{4m}(2m(j-r)+k)^2}
\\[2mm]
& & 
\times \, 
\big\{e^{\pi i(2mr-k)(z_1+z_3)}+e^{-\pi i(2mr-k)(z_1+z_3)}\big\}
\big[\theta_{k,m}-\theta_{-k,m}](\tau, z_1-z_3)
\\[2mm]
& &
+ \, 2 \, \sum_{j=1}^{\infty} \sum_{r=0}^{j-1} 
\sum_{\substack{k=1 \\[1mm] k \, : \, {\rm odd}}}^{m-1}
(-1)^j q^{(m+1)j^2-\frac{1}{4m}(2m(j-r)-k)^2}
\\[2mm]
& & 
\times 
\big\{e^{\pi i(2mr+k)(z_1+z_3)}+e^{-\pi i(2mr+k)(z_1+z_3)}\big\}
\big[\theta_{k,m}-\theta_{-k,m}](\tau, z_1-z_3)
\end{eqnarray*}}

\vspace{-3mm}

\noindent
Next we compute the RHS of \eqref{n3-2:eqn:2022-316e} in the case when 
$z_2=z_1+\frac12$ :
{\allowdisplaybreaks
\begin{eqnarray*}
& &\hspace{-10mm}
{\rm (III)} \,\ \overset{\substack{put \\[0.7mm]}}{:=} \,\ 
\text{RHS of \eqref{n3-2:eqn:2022-316e}}\big|_{z_2=z_1+\frac12}
\\[2mm]
&=&
i \, \theta_{0,m+1}\Big(\tau, \frac{2z_1+\frac12+m(z_1-z_3)}{m+1}\Big)
\frac{\eta(\tau)^3 \, \vartheta_{11}(\tau, -\frac12)}{
\vartheta_{11}(\tau, z_1)\vartheta_{11}(\tau, z_1+\frac12)}
\\[2mm]
& &
+ \, i \, \theta_{0,m+1}
\Big(\tau, \frac{-\frac12-2z_3+m(z_1-z_3)}{m+1}\Big)
\frac{\eta(\tau)^3 \, \vartheta_{11}(\tau, -\frac12)}{
\vartheta_{11}(\tau, z_1)\vartheta_{11}(\tau, -\frac12-z_3)}
\\[2mm]
&=&
2i \, \eta(\tau)^2\eta(2\tau)^2 \, \Bigg\{
- \,\ \frac{
\theta_{0,m+1}\big(\tau, \, \frac{\frac12+z_1+z_3}{m+1}+z_1-z_3\big)
}{\vartheta_{11}(\tau, z_1) \vartheta_{10}(\tau, z_1)}
\\[2mm]
& & \hspace{26mm}
+ \,\ \frac{
\theta_{0,m+1}\big(\tau, \, -\frac{\frac12+z_1+z_3}{m+1}+z_1-z_3\big)
}{\vartheta_{11}(\tau, z_3) \vartheta_{10}(\tau, z_3)} \Bigg\}
\\[2mm]
&=&
2i \, \eta(2\tau)^3 \, \Bigg\{
- \,\ \frac{
\theta_{0,m+1}\big(\tau, \, \frac{\frac12+z_1+z_3}{m+1}+z_1-z_3\big)
}{\vartheta_{11}(2\tau, 2z_1)}
\\[2mm]
& & \hspace{26mm}
+ \,\ \frac{
\theta_{0,m+1}\big(\tau, \, -\frac{\frac12+z_1+z_3}{m+1}+z_1-z_3\big)
}{\vartheta_{11}(2\tau, 2z_3)} \Bigg\}
\end{eqnarray*}}
where we used \, $\vartheta_{11}(\tau, -\frac12) 
=2\frac{\eta(2\tau)^2}{\eta(\tau)}$ \, and \, 
$\vartheta_{11}(\tau, z+\frac12) =-\vartheta_{10}(\tau,z)$ \, and
$$
\vartheta_{11}(\tau, z)\vartheta_{10}(\tau, z) \, = \, 
\frac{\eta(\tau)^2}{\eta(2\tau)} \, \vartheta_{11}(2\tau, 2z) \, .
$$

\medskip

Then, by ${\rm (I)}={\rm (III)}-{\rm (II)}$, we obtain 
\eqref{n3-2:eqn:2022-317c}, 
which completes proof of Proposition \ref{n3-2:prop:2022-314a}.
\end{proof}

\begin{cor} 
\label{n3-2:cor:2022-314a}
For $m \in \frac12 \nnn$ and $s \in \frac12 \nnn_{\rm odd}$
the following formulas hold:
\begin{enumerate}
\item[{\rm 1)}] \quad $\theta_{0, 2m+1}
\Big(\tau, \dfrac{-\frac12+2m(z_1-z_2)}{2m+1} \Big) \, 
\Phi^{[m, \frac12]}(2\tau, \, 2z_1, \, 2z_2, \, 0)$
{\allowdisplaybreaks
\begin{eqnarray}
& & \hspace{-10mm}
= \,\ 
-i \, \eta(2\tau)^3 \, \Bigg\{ \, 
\frac{\displaystyle 
\theta_{0, 2m+1}\Big(\tau, \, \frac{\frac12+z_1-z_2}{2m+1}+ z_1+z_2\Big)
}{\vartheta_{11}(2\tau, 2z_1)} 
\nonumber
\\[2mm]
& & \hspace{15mm}
+ \, 
\frac{\displaystyle 
\theta_{0, 2m+1}\Big(\tau, \, \frac{\frac12+z_1-z_2}{2m+1}- z_1-z_2\Big)
}{\vartheta_{11}(2\tau, 2z_2)} \Bigg\}
\nonumber
\\[1mm]
& & \hspace{-10mm}
+ \, \sum_{j=1}^{\infty} \sum_{r=1}^j 
\sum_{\substack{k=1 \\[1mm] k \, : \, {\rm odd}}}^{2m-1}
(-1)^j q^{(2m+1)j^2-\frac{1}{8m}(4m(j-r)+k)^2}
\nonumber
\\[2mm]
& & \hspace{-7mm}
\times \, 
\big\{e^{\pi i(4mr-k)(z_1-z_2)}+e^{-\pi i(4mr-k)(z_1-z_2)}\big\}
\big[\theta_{k,2m}-\theta_{-k,2m}](\tau, z_1+z_2)
\nonumber
\\[2mm]
& & \hspace{-10mm}
- \, \sum_{j=1}^{\infty} \sum_{r=0}^{j-1} 
\sum_{\substack{k=1 \\[1mm] k \, : \, {\rm odd}}}^{2m-1}
(-1)^j q^{(2m+1)j^2-\frac{1}{8m}(4m(j-r)-k)^2}
\nonumber
\\[2mm]
& & \hspace{-7mm}
\times \, 
\big\{e^{\pi i(4mr+k)(z_1-z_2)}+e^{-\pi i(4mr+k)(z_1-z_2)}\big\}
\big[\theta_{k,2m}-\theta_{-k,2m}](\tau, z_1+z_2)
\nonumber
\\[0mm]
& &
\label{n3-2:eqn:2022-325a}
\end{eqnarray}}
\item[{\rm 2)}] \quad $\theta_{0, 2m+1}
\Big(\tau, \dfrac{-\frac12+2m(z_1-z_2)}{2m+1} \Big) \, 
\Phi^{[m, s]}(2\tau, \, 2z_1, \, 2z_2, \, 0)$
{\allowdisplaybreaks
\begin{eqnarray*}
& & \hspace{-5mm}
= \,\ 
-i \, \eta(2\tau)^3 \Bigg\{ 
\frac{\displaystyle 
\theta_{0, 2m+1}\Big(\tau, \, \frac{\frac12+z_1-z_2}{2m+1}+ z_1+z_2\Big)
}{\vartheta_{11}(2\tau, 2z_1)} 
\\[2mm]
& & \hspace{20mm}
+ \,\ 
\frac{\displaystyle 
\theta_{0, 2m+1}\Big(\tau, \, \frac{\frac12+z_1-z_2}{2m+1}- z_1-z_2\Big)
}{\vartheta_{11}(2\tau, 2z_2)} \Bigg\}
\\[1mm]
& & \hspace{-5mm}
+ \, \sum_{j=1}^{\infty} \sum_{r=1}^j 
\sum_{\substack{k=1 \\[1mm] k \, : \, {\rm odd}}}^{2m-1}
(-1)^j q^{(2m+1)j^2-\frac{1}{8m}(4m(j-r)+k)^2}
\\[2mm]
& & 
\times \, 
\big\{e^{\pi i(4mr-k)(z_1-z_2)}+e^{-\pi i(4mr-k)(z_1-z_2)}\big\}
\big[\theta_{k,2m}-\theta_{-k,2m}](\tau, z_1+z_2)
\\[2mm]
& & \hspace{-5mm}
- \, \sum_{j=1}^{\infty} \sum_{r=0}^{j-1} 
\sum_{\substack{k=1 \\[1mm] k \, : \, {\rm odd}}}^{2m-1}
(-1)^j q^{(2m+1)j^2-\frac{1}{8m}(4m(j-r)-k)^2}
\\[1mm]
& & \hspace{-2mm}
\times 
\big\{e^{\pi i(4mr+k)(z_1-z_2)}+e^{-\pi i(4mr+k)(z_1-z_2)}\big\}
\big[\theta_{k,2m}-\theta_{-k,2m}](\tau, z_1+z_2)
\\[3mm]
& & \hspace{-5mm}
- \,\ \theta_{0, 2m+1}\Big(\tau, \frac{-\frac12+2m(z_1-z_2)}{2m+1} \Big) 
\\[0mm]
& & \hspace{-2mm}
\times \sum_{k=0}^{s-\frac32}
e^{\pi i (1+2k)(z_1-z_2)} \, 
q^{- \, \frac{(1+2k)^2}{8m}} \, 
\big[\theta_{\frac12+k, m} \, - \, \theta_{-(\frac12+k), m}\big]
(2\tau, \, 2(z_1+z_2))
\end{eqnarray*}}
\end{enumerate}
\end{cor}

\begin{proof} By Lemma 2.5 in \cite{W2022} and \eqref{n3-2:eqn:2022-331d}, 
we have
{\allowdisplaybreaks
\begin{eqnarray}
& & \hspace{-7mm}
\Phi^{[m, \frac12]}(2\tau, \, 2z_1, \, 2z_2, \, 0)
\nonumber
\\[2mm]
&=&
\frac12 \, \bigg\{
\Phi^{[2m,1]}(\tau, z_1,z_2,0) \, - \, 
\Phi^{[2m,1]}\Big(\tau, z_1+\frac12,z_2-\frac12,0\Big) \bigg\}
\nonumber
\\[2mm]
&=&
\frac12 \, \bigg\{
\Phi^{[2m,0]}(\tau, z_1,z_2,0) \, - \, 
\Phi^{[2m,0]}\Big(\tau, z_1+\frac12,z_2-\frac12,0\Big) \bigg\}
\label{n3-2:eqn:2022-317e}
\end{eqnarray}}
Then 1) follows from \eqref{n3-2:eqn:2022-317e} and Proposition 
\ref{n3-2:prop:2022-314a}, and 2) follows from 1) and 
Lemma \ref{n3-2:lemma:2022-312a}.
\end{proof}

We note that, by using the function $\theta^{(-)}_{k,m}(\tau, z)$ 
defined by
\begin{equation}
\theta^{(-)}_{k,m}(\tau, z) \,\ := \,\ 
\sum_{j \in \zzz}(-1)^j \, e^{2\pi im(j+\frac{k}{2m})z} \, 
q^{m(j+\frac{k}{2m})^2} \, ,
\label{n3-2:eqn:2022-325b}
\end{equation}
the formula \eqref{n3-2:eqn:2022-325a} is rewritten as follows:
{\allowdisplaybreaks
\begin{eqnarray}
& &
\theta_{0, 2m+1}^{(-)}
\Big(\tau, \dfrac{2m(z_1-z_2)}{2m+1} \Big) \, 
\Phi^{[m, \frac12]}(2\tau, \, 2z_1, \, 2z_2, \, 0)
\nonumber
\\[2mm]
& & \hspace{-5mm}
= 
-i \eta(2\tau)^3 \Bigg\{
\frac{\displaystyle 
\theta_{0, 2m+1}^{(-)}\Big(\tau, \, \frac{z_1-z_2}{2m+1}+ z_1+z_2\Big)
}{\vartheta_{11}(2\tau, 2z_1)} 
+ 
\frac{\displaystyle 
\theta_{0, 2m+1}^{(-)}\Big(\tau, \, \frac{z_1-z_2}{2m+1}- z_1-z_2\Big)
}{\vartheta_{11}(2\tau, 2z_2)} \Bigg\}
\nonumber
\\[2mm]
& & \hspace{-5mm}
+ \, \sum_{j=1}^{\infty} \sum_{r=1}^j 
\sum_{\substack{k=1 \\[1mm] k \, : \, {\rm odd}}}^{2m-1}
(-1)^j q^{(2m+1)j^2-\frac{1}{8m}(4m(j-r)+k)^2}
\nonumber
\\[2mm]
& & 
\times \, 
\big\{e^{\pi i(4mr-k)(z_1-z_2)}+e^{-\pi i(4mr-k)(z_1-z_2)}\big\}
\big[\theta_{k,2m}-\theta_{-k,2m}](\tau, z_1+z_2)
\nonumber
\\[2mm]
& & \hspace{-5mm}
- \, \sum_{j=1}^{\infty} \sum_{r=0}^{j-1} 
\sum_{\substack{k=1 \\[1mm] k \, : \, {\rm odd}}}^{2m-1}
(-1)^j q^{(2m+1)j^2-\frac{1}{8m}(4m(j-r)-k)^2}
\nonumber
\\[2mm]
& & 
\times 
\big\{e^{\pi i(4mr+k)(z_1-z_2)}+e^{-\pi i(4mr+k)(z_1-z_2)}\big\}
\big[\theta_{k,2m}-\theta_{-k,2m}](\tau, z_1+z_2)
\nonumber
\\[0mm]
& &
\label{n3-2:eqn:2022-325c}
\end{eqnarray}}

By Corollary \ref{n3-2:cor:2022-314a} and Lemma \ref{n3-2:lemma:2022-312a},
we obtain an explicit formula for \\
$\Phi^{[m,s]}(\tau, z_1, z_2,t)$ when $s \in \frac12 \zzz_{\rm odd}$. 
We now consider $\Phi^{[m,s]}$ when $s \in \zzz$. For this sake, 
we note the following lemma:

\begin{lemma} \quad 
\label{n3-2:lemma:2022-324a}
For $m \in \frac12\nnn_{\rm odd}$, the following formula holds:
{\allowdisplaybreaks
\begin{eqnarray*}
& & \hspace{-12mm}
\Phi^{[m, 0]}(\tau, z_1,z_2,0)
\,\ = \,\ 
e^{2\pi imz_1} \, \Phi^{[m, \frac12]}(\tau, z_1,z_2+\tau,0)
\\[2mm]
& &
+ \, 
\sum_{k=1}^{m-\frac12}e^{\pi ik(z_1-z_2)} \, 
q^{-\frac{k^2}{4m}} \, 
\big[\theta_{k,m}-\theta_{-k,m}\big](\tau, z_1+z_2)
\end{eqnarray*}}
\end{lemma}

\begin{proof} Letting $s=\frac12$ and $p=1$ in Lemma 
\ref{n3-2:lemma:2022-324b}, one has
\begin{subequations}
\begin{equation}
\Phi^{[m, \frac12+m]}(\tau, z_1,z_2,t)
\,\ = \,\ 
e^{2\pi imz_1} \, \Phi^{[m, \frac12]}(\tau, z_1,z_2+\tau,t)
\label{n3-2:eqn:2022-324a1}
\end{equation}
And, by letting $s=0$ and $j=\frac12+m$ in Lemma \ref{n3-2:lemma:2022-312a},
one has
{\allowdisplaybreaks
\begin{eqnarray}
& & \hspace{-12mm}
\Phi^{[m, 0]}(\tau, z_1,z_2,0)
\,\ = \,\ 
\Phi^{[m, \frac12+m]}(\tau, z_1,z_2,0)
\nonumber
\\[3mm]
& &
+ \, 
\sum_{k=1}^{\frac12+m-1}e^{\pi ik(z_1-z_2)} \, 
q^{-\frac{k^2}{4m}} \, 
\big[\theta_{k,m}-\theta_{-k,m}\big](\tau, z_1+z_2)
\label{n3-2:eqn:2022-324a2}
\end{eqnarray}}
\end{subequations}
Then Lemma \ref{n3-2:lemma:2022-324a} follows from 
\eqref{n3-2:eqn:2022-324a1} and \eqref{n3-2:eqn:2022-324a2}.
\end{proof}

Then, by Corollary \ref{n3-2:cor:2022-314a} and Lemma 
\ref{n3-2:lemma:2022-324a}, we have the following:

\begin{prop}
\label{n3-2:prop:2022-331a}
For $m \in \frac12 \nnn_{\rm odd}$, 
$\Phi^{[m,0]}$ is given by the following formula:
{\allowdisplaybreaks
\begin{eqnarray*}
& & \hspace{-8mm}
\theta_{-2m, 2m+1}^{(-)}\Big(\tau, \, \dfrac{2m(z_1-z_2)}{2m+1}\Big) \, 
\Phi^{[m, 0]}(2\tau, 2z_1,2z_2,0)
\\[3mm]
& & \hspace{-10mm}
= \,\ 
- \, i \, \eta(2\tau)^3 \, \Bigg\{ 
\frac{
\theta_{2m, 2m+1}^{(-)}\Big(\tau, \, \dfrac{z_1-z_2}{2m+1}+z_1+z_2\Big)
}{\vartheta_{11}(2\tau, 2z_1)} 
\\[2mm]
& & \hspace{5mm}
+ \,\ 
\frac{
\theta_{2m, 2m+1}^{(-)}\Big(\tau, \,\ \dfrac{z_1-z_2}{2m+1}-z_1-z_2\Big)
}{\vartheta_{11}(2\tau, 2z_2)} \Bigg\}
\\[3mm]
& & \hspace{-5mm}
- \,\ 
q^{-\frac{m}{2(2m+1)}} \, 
e^{\frac{2\pi im}{2m+1}(z_1-z_2)} \, 
\sum_{j=1}^{\infty} \sum_{r=1}^j 
\sum_{\substack{k=0 \\[1mm] k \, : \, {\rm even}}}^{2m-1}
(-1)^j q^{(2m+1)j^2-\frac{1}{8m}(4m(j-r)+2m-k)^2} 
\\[2mm]
& & 
\times \,\ \big\{ 
e^{\pi i(4mr-2m+k)(z_1-z_2)} q^{-\frac12(4mr-2m+k)}
\\[2mm]
& & \hspace{5mm} 
+ \, 
e^{-\pi i(4mr-2m+k)(z_1-z_2)} q^{\frac12(4mr-2m+k)}\big\} \, 
\big[\theta_{k,2m}- \, \theta_{-k,2m}](\tau, z_1+z_2)
\\[2mm]
& & \hspace{-5mm}
+ \,\ 
q^{-\frac{m}{2(2m+1)}} \, 
e^{\frac{2\pi im}{2m+1}(z_1-z_2)} \, 
\sum_{j=1}^{\infty} \sum_{r=0}^{j-1} 
\sum_{\substack{k=0 \\[1mm] k \, : \, {\rm even}}}^{2m-1}
(-1)^j q^{(2m+1)j^2-\frac{1}{8m}(4m(j-r)-2m+k)^2} 
\\[2mm]
& & 
\times \,\ \big\{
e^{\pi i(4mr+2m-k)(z_1-z_2)} q^{-\frac12(4mr+2m-k)}
\\[2mm]
& & \hspace{5mm}
+ \, 
e^{-\pi i(4mr+2m-k)(z_1-z_2)} q^{\frac12(4mr+2m-k)}
\big\} \, 
\big[\theta_{k,2m}- \, \theta_{-k,2m}](\tau, z_1+z_2)
\\[3mm]
& & \hspace{-5mm}
+ \,\ 
\theta_{-2m, 2m+1}^{(-)}\Big(\tau, \, \dfrac{2m(z_1-z_2)}{2m+1}\Big) 
\\[2mm]
& & 
\times \, 
\sum_{k=1}^{m-\frac12}e^{2\pi ik(z_1-z_2)} \, 
q^{-\frac{k^2}{2m}} \, 
\big[\theta_{k,m}-\theta_{-k,m}\big](2\tau, 2z_1+2z_2)
\end{eqnarray*}}
\end{prop}

\begin{proof} By Lemma \ref{n3-2:lemma:2022-324a}, we have
\begin{subequations}
{\allowdisplaybreaks
\begin{eqnarray}
& & \hspace{-12mm}
\Phi^{[m, 0]}(2\tau, 2z_1,2z_2,0)
\,\ = \,\ 
e^{4\pi imz_1} \, \Phi^{[m, \frac12]}(2\tau, 2z_1,2z_2+2\tau,0)
\nonumber
\\[2mm]
& &
+ \, 
\sum_{k=1}^{m-\frac12}e^{2\pi ik(z_1-z_2)} \, 
q^{-\frac{k^2}{2m}} \, 
\big[\theta_{k,m}-\theta_{-k,m}\big](2\tau, 2z_1+2z_2)
\label{n3-2:eqn:2022-331c1}
\end{eqnarray}}
Letting \, $z_2 \rightarrow z_2+\tau$ \, in \eqref{n3-2:eqn:2022-325c}, 
we have
{\allowdisplaybreaks
\begin{eqnarray}
& &
\theta_{0, 2m+1}^{(-)}
\Big(\tau, \dfrac{2m(z_1-z_2-\tau)}{2m+1} \Big) \, 
\Phi^{[m, \frac12]}(2\tau, \, 2z_1, \, 2z_2+2\tau, \, 0)
\nonumber
\\[2mm]
& & \hspace{-5mm}
= 
-i \, \eta(2\tau)^3 \Bigg\{
\frac{\displaystyle 
\theta_{0, 2m+1}^{(-)}
\Big(\tau, \, \frac{z_1-z_2-\tau}{2m+1}+ z_1+z_2+\tau\Big)
}{\vartheta_{11}(2\tau, 2z_1)} 
\nonumber
\\[1mm]
& & \hspace{15mm}
+ \,\ 
\frac{\displaystyle 
\theta_{0, 2m+1}^{(-)}
\Big(\tau, \, \frac{z_1-z_2-\tau}{2m+1}- z_1-z_2-\tau\Big)
}{\vartheta_{11}(2\tau, 2z_2+2\tau)} \Bigg\}
\nonumber
\\[2mm]
& & \hspace{-7mm}
+ \, \sum_{j=1}^{\infty} \sum_{r=1}^j 
\sum_{\substack{k=1 \\[1mm] k \, : \, {\rm odd}}}^{2m-1}
(-1)^j q^{(2m+1)j^2-\frac{1}{8m}(4m(j-r)+k)^2}
\nonumber
\\[1mm]
& & \hspace{-7mm}
\times 
\big\{e^{\pi i(4mr-k)(z_1-z_2-\tau)}+e^{-\pi i(4mr-k)(z_1-z_2-\tau)}\big\}
\big[\theta_{k,2m}-\theta_{-k,2m}](\tau, z_1+z_2+\tau)
\nonumber
\\[2mm]
& & \hspace{-7mm}
- \, \sum_{j=1}^{\infty} \sum_{r=0}^{j-1} 
\sum_{\substack{k=1 \\[1mm] k \, : \, {\rm odd}}}^{2m-1}
(-1)^j q^{(2m+1)j^2-\frac{1}{8m}(4m(j-r)-k)^2}
\nonumber
\\[1mm]
& & \hspace{-7mm}
\times 
\big\{e^{\pi i(4mr+k)(z_1-z_2-\tau)}+e^{-\pi i(4mr+k)(z_1-z_2-\tau)}\big\}
\big[\theta_{k,2m}-\theta_{-k,2m}](\tau, z_1+z_2+\tau)
\nonumber
\\[0mm]
& &
\label{n3-2:eqn:2022-331c2}
\end{eqnarray}}
\end{subequations}
Rewriting \eqref{n3-2:eqn:2022-331c2} by using the formulas
{\allowdisplaybreaks
\begin{eqnarray*}
& &\theta_{0,2m+1}^{(-)}
\Big(\tau, \,\ \dfrac{2m(z-\tau)}{2m+1}\Big)
\,\ = \,\ 
q^{-\frac{m^2}{2m+1}} \, 
e^{\frac{4\pi im^2 z}{2m+1}} \, 
\theta_{-2m,2m+1}^{(-)}\Big(\tau, \,\ \dfrac{2mz}{2m+1}\Big)
\\[2mm]
& &
\theta_{0,2m+1}^{(-)}
\Big(\tau, \, \dfrac{z-\tau}{2m+1}+z'+\tau\Big)
\\[2mm]
& & \hspace{15mm}
= \,\ 
q^{-\frac{m^2}{2m+1}} \, 
e^{-2\pi im (\frac{z}{2m+1}+z')} \, 
\theta_{2m,2m+1}^{(-)}\Big(\tau, \, \dfrac{z}{2m+1}+z'\Big)
\\[2mm]
& &
\theta_{0,2m+1}^{(-)}
\Big(\tau, \, \dfrac{z-\tau}{2m+1}-z'-\tau\Big)
\\[2mm]
& & \hspace{15mm}
= \,\ 
- \, q^{-\frac{m^2}{2m+1}} 
e^{2\pi i (m+1) (\frac{z}{2m+1}-z')} 
\theta_{2m,2m+1}^{(-)}\Big(\tau, \, \dfrac{z}{2m+1}-z'\Big)
\end{eqnarray*}}
and substituting it into \eqref{n3-2:eqn:2022-331c1}, we obtain 
Proposition \ref{n3-2:prop:2022-331a}.
\end{proof}

By Proposition \ref{n3-2:prop:2022-331a} and Lemma 
\ref{n3-2:lemma:2022-312a}, we obtain an explicit formula for \\
$\Phi^{[m,s]}(\tau, z_1, z_2,t)$ when $m \in \frac12 \nnn_{\rm odd}$
and $s \in \zzz$. To go further, we note the following:

\begin{prop} 
\label{n3-2:prop:2022-331c}
For $m \in \frac12 \nnn$, the following formula holds:
\begin{equation}
\Phi^{[2m,0]}(\tau, z_1, z_2,t) \, = \, 
\Phi^{[m,0]}(2\tau, 2z_1, 2z_2,2t)+
\Phi^{[m,\frac12]}(2\tau, 2z_1, 2z_2,2t)
\label{n3-2:eqn:2022-403a}
\end{equation}
\end{prop}

\begin{proof} This formula is obtained immediately from 
Lemma 2.5 in \cite{W2022} and \eqref{n3-2:eqn:2022-331d}.
\end{proof}

\begin{prop}
\label{n3-2:prop:2022-403b}
An explicit formula for the function $\Phi^{[m,s]}$ can be known 
for all $m$ and $s$.
\end{prop}

\begin{proof} 
Since $\Phi^{[m,s]}$, when $(m,s) \in 
(\frac12 \nnn_{\rm odd} \times \frac12 \zzz) \cup 
(\frac12 \nnn \times \frac12 \zzz_{\rm odd})$, 
is known by Corollary \ref{n3-2:cor:2022-314a} and 
Proposition \ref{n3-2:prop:2022-331a} and 
Lemma \ref{n3-2:lemma:2022-312a},
it suffices to prove the claim:
\begin{subequations}
\begin{equation}
\text{\lq \lq an explicit formula for} \,\ \Phi^{[m, 0]} \,\ (m \in \nnn) \,\ 
\text{can be known"} \, .
\label{n3-2:eqn:2022-403b}
\end{equation}
Decomposing $m=2^p m'$, where $m' \in \nnn_{\rm odd}$ and 
$p \in \zzz_{\geq 0}$, we shall show the claim:
\begin{equation}
\text{\lq \lq an explicit formula for} \,\ \Phi^{[2^pm', 0]} \,\ 
\text{can be known"}
\label{n3-2:eqn:2022-403c}
\end{equation}
\end{subequations}
by induction on $p$. As the first step, we consider the case $p=0$. 
In the case $p=0$, the formula \eqref{n3-2:eqn:2022-403a} gives
\begin{equation}
\Phi^{[m',0]}(\tau, z_1, z_2,t) \, = \, 
\Phi^{[\frac{m'}{2},0]}(2\tau, 2z_1, 2z_2,2t)+
\Phi^{[\frac{m'}{2},\frac12]}(2\tau, 2z_1, 2z_2,2t)
\label{n3-2:eqn:2022-403d}
\end{equation}
Since the RHS of \eqref{n3-2:eqn:2022-403d} is known by 
Proposition \ref{n3-2:prop:2022-331a} and 
Corollary \ref{n3-2:cor:2022-314a}, $\Phi^{[m',0]}$ is known 
so ${\rm \eqref{n3-2:eqn:2022-403c}}_{p=0}$ holds.
As the second step, we assume that
${\rm \eqref{n3-2:eqn:2022-403c}}_{p}$ is true and apply 
\eqref{n3-2:eqn:2022-403a} to $m=2^{p}m'$:
\begin{equation}
\Phi^{[2^{p+1}m',0]}(\tau, z_1, z_2,t) = 
\Phi^{[2^pm',0]}(2\tau, 2z_1, 2z_2,2t)+
\Phi^{[2^pm',\frac12]}(2\tau, 2z_1, 2z_2,2t)
\label{n3-2:eqn:2022-403e}
\end{equation}
where the RHS of \eqref{n3-2:eqn:2022-403e} is known by assumption.
So the LHS of \eqref{n3-2:eqn:2022-403e} can be known, 
namely ${\rm \eqref{n3-2:eqn:2022-403c}}_{p+1}$ is true, 
proving Proposition \ref{n3-2:prop:2022-403b}.
\end{proof}

\section{Characters of N=3 modules}

\begin{prop}
\label{n3-2:prop:2022-320a}
For $m \in \nnn$, the character and the supercharacter of the N=3 module 
$H(\Lambda^{[K(m), 0]})$ are given by the following formulas:
\begin{enumerate}
\item[{\rm 1)}] \,\ ${\rm ch}_{H(\Lambda^{[K(m), 0]})}^{(+)}(\tau,z)$ 
{\allowdisplaybreaks
\begin{eqnarray*}
& & \hspace{-15mm}
= -i \,
e^{-\frac{\pi im}{4}} 
\frac{\eta(2\tau)^2}{\eta(\frac{\tau}{2})} \cdot
\frac{\vartheta_{00}(\tau, z)}{\vartheta_{11}(\tau, z)} \cdot 
\frac{1}{\theta_{-\frac{m}{2}, m+1}(\tau, \frac12)}
\, \Bigg\{ \, 
\frac{\theta_{\frac12, m+1}(\tau,z)}{\theta_{-\frac12, 1}(\tau,z)}
- 
\frac{\theta_{-\frac12, m+1}(\tau,z)}{\theta_{\frac12, 1}(\tau,z)}
\Bigg\}
\\[2mm]
& & \hspace{-14mm}
+ \,\ 
q^{\frac{m^2}{16(m+1)}} \, e^{-\frac{\pi im}{4}} \, 
\frac{1}{\eta(\frac{\tau}{2})\eta(2\tau)} \cdot
\frac{\vartheta_{00}(\tau, z)}{\vartheta_{11}(\tau, z)} \cdot
\frac{1}{\theta_{-\frac{m}{2}, m+1}(\tau, \frac12)}
\\[2mm]
& &\hspace{-10mm}
\times \, \Bigg\{
\sum_{j=1}^{\infty} \sum_{r=1}^j 
\sum_{\substack{k=1 \\[1mm] k \, : \, {\rm odd}}}^{m-1}
(-1)^j \, q^{(m+1)j^2-\frac{1}{4m}(2mr-k)^2-\frac{m}{16}} %
\\[2mm]
& & \hspace{-7mm}
\times \big[
q^{(j+\frac14) (2mr-k)}e^{\frac{\pi i}{2}(2mr+k)}
+
q^{(j-\frac14) (2mr-k)}e^{\frac{\pi i}{2}(2mr-k)}
\big]
\big[\theta_{k,m}-\theta_{-k,m}](\tau, z)
\\[3mm]
& & \hspace{-10mm}
- \, \sum_{j=1}^{\infty} \sum_{r=0}^{j-1} 
\sum_{\substack{k=1 \\[1mm] k \, : \, {\rm odd}}}^{m-1}
(-1)^j \, q^{(m+1)j^2-\frac{1}{4m}(2mr+k)^2-\frac{m}{16}} 
\\[0mm]
& & \hspace{-7mm}
\times 
\big[
q^{(j+\frac14) (2mr+k)}e^{\frac{\pi i}{2}(2mr-k)}
+ 
q^{(j-\frac14) (2mr+k)}e^{\frac{\pi i}{2}(2mr+k)}
\big]
\big[\theta_{k,m}-\theta_{-k,m}](\tau, z)
\Bigg\}
\end{eqnarray*}}

\vspace{-3mm}

\item[{\rm 2)}] \,\ ${\rm ch}_{H(\Lambda^{[K(m), 0]})}^{(-)}(\tau,z)$ 
{\allowdisplaybreaks
\begin{eqnarray*}
& & \hspace{-10mm}
= \, 
\frac{\eta(\frac{\tau}{2})\eta(2\tau)^3}{\eta(\tau)^3} \cdot
\frac{\vartheta_{01}(\tau,z)}{\vartheta_{11}(\tau,z)} \cdot 
\frac{1}{\theta_{\frac{m}{2}, m+1}^{(-)}(\tau,0)} \, 
\Bigg\{ \, 
\frac{\theta_{\frac12, m+1}^{(-)}(\tau,z)}{\theta_{-\frac12, 1}^{(-)}(\tau,z)}
\, - \, 
\frac{\theta_{-\frac12, m+1}^{(-)}(\tau,z)}{\theta_{\frac12, 1}^{(-)}(\tau,z)}
\Bigg\}
\\[2mm]
& & \hspace{-5mm}
+ \,\ q^{\frac{m^2}{16(m+1)}} \, 
\frac{\eta(\frac{\tau}{2})}{\eta(\tau)^3} \cdot
\frac{\vartheta_{01}(\tau,z)}{\vartheta_{11}(\tau,z)} \cdot 
\frac{1}{\theta_{\frac{m}{2}, m+1}^{(-)}(\tau,0)}
\\[2mm]
& &
\times \,\ \Bigg\{
\sum_{j=1}^{\infty} \sum_{r=1}^j 
\sum_{\substack{k=1 \\[1mm] k \, : \, {\rm odd}}}^{m-1}
(-1)^j \, q^{(m+1)j^2-\frac{1}{4m}(2mr-k)^2-\frac{m}{16}}
\\[2mm]
& & \hspace{5mm}
\times \, \big\{q^{(j+\frac14)(2mr-k)}
\, + \, q^{(j-\frac14)(2mr-k)}
\big\}
\big[\theta_{k,m}-\theta_{-k,m}](\tau, z)
\\[3mm]
& & 
- \,\ \sum_{j=1}^{\infty} \sum_{r=0}^{j-1} 
\sum_{\substack{k=1 \\[1mm] k \, : \, {\rm odd}}}^{m-1}
(-1)^j \, q^{(m+1)j^2-\frac{1}{4m}(2mr+k)^2-\frac{m}{16}}
\\[0mm]
& & \hspace{5mm}
\times \, \big\{
q^{(j+\frac14)(2mr+k)} \, + \, q^{(j-\frac14)(2mr+k)}
\big\}
\big[\theta_{k,m}-\theta_{-k,m}](\tau, z)
\Bigg\}
\end{eqnarray*}}
\end{enumerate}
\end{prop}

\begin{proof} By Proposition 4.1 in \cite{W2022}, the character and 
supercharacter of the N=3 module $H(\Lambda^{[K(m), m_2]})$ are given by
the following formulas:
\begin{subequations}
{\allowdisplaybreaks
\begin{eqnarray}
\big(\overset{N=3}{R}{}^{(+)}{\rm ch}_{H(\Lambda^{[K(m), m_2]})}^{(+)}\big)
(\tau,z) \hspace{-2mm} &=& \hspace{-2mm} 
q^{-\frac{m}{16}} \, 
\Phi^{[\frac{m}{2}, \frac{m_2+1}{2}]}
\Big(2\tau, z+\frac{\tau}{2}-\frac12, z-\frac{\tau}{2}+\frac12, 0\Big)
\nonumber
\\[2mm]
\big(\overset{N=3}{R}{}^{(-)}{\rm ch}_{H(\Lambda^{[K(m), m_2]})}^{(-)}\big)
(\tau,z) \hspace{-2mm} &=& \hspace{-2mm}
q^{-\frac{m}{16}} \, 
\Phi^{[\frac{m}{2}, \frac{m_2+1}{2}]}
\Big(2\tau, z+\frac{\tau}{2}, z-\frac{\tau}{2}, 0\Big)
\label{n3-2:eqn:2022-321b1}
\end{eqnarray}}
where 
\begin{equation} \left\{
\begin{array}{ccl}
\overset{N=3}{R}{}^{(+)}(\tau,z) &:=& \eta(\frac{\tau}{2})\eta(2\tau)
\dfrac{\vartheta_{11}(\tau,z)}{\vartheta_{00}(\tau,z)}
\\[4mm]
\overset{N=3}{R}{}^{(-)}(\tau,z) &:=& \dfrac{\eta(\tau)^3}{\eta(\frac{\tau}{2})} 
\cdot
\dfrac{\vartheta_{11}(\tau,z)}{\vartheta_{01}(\tau,z)}
\end{array} \right.
\label{n3-2:eqn:2022-321b2}
\end{equation}
\end{subequations}
In particular when $m_2=0$, we have
{\allowdisplaybreaks
\begin{eqnarray}
\big(\overset{N=3}{R}{}^{(+)}{\rm ch}_{H(\Lambda^{[K(m), 0]})}^{(+)}\big)
(\tau,z) \hspace{-2mm}&=& \hspace{-2mm}
q^{-\frac{m}{16}} \, 
\Phi^{[\frac{m}{2}, \frac{1}{2}]}
\Big(2\tau, z+\frac{\tau}{2}-\frac12, z-\frac{\tau}{2}+\frac12, 0\Big)
\nonumber
\\[2mm]
\big(\overset{N=3}{R}{}^{(-)}{\rm ch}_{H(\Lambda^{[K(m), 0]})}^{(-)}\big)
(\tau,z) \hspace{-2mm} &=& \hspace{-2mm}
q^{-\frac{m}{16}} \, 
\Phi^{[\frac{m}{2}, \frac{1}{2}]}
\Big(2\tau, z+\frac{\tau}{2}, z-\frac{\tau}{2}, 0\Big) \, .
\label{n3-2:eqn:2022-321c}
\end{eqnarray}}

\vspace{-1mm}

\noindent
The functions $\Phi^{[\frac{m}{2}, \frac{1}{2}]}
(2\tau, z+\frac{\tau}{2}-\frac12, z-\frac{\tau}{2}+\frac12, 0)$
and 
$\Phi^{[\frac{m}{2}, \frac{1}{2}]}
(2\tau, z+\frac{\tau}{2}, z-\frac{\tau}{2}, 0)$
in the RHS of \eqref{n3-2:eqn:2022-321c} 
are obtained from Corollary \ref{n3-2:cor:2022-314a} by 
replacing $m$ with $\frac{m}{2}$ and putting $\left\{
\begin{array}{ccc}
z_1 &=&\frac{z}{2}+\frac{\tau}{4}-\frac14 \\[1mm]
z_2 &=&\frac{z}{2}-\frac{\tau}{4}+\frac14
\end{array}\right. $ and $\left\{
\begin{array}{ccc}
z_1 &=&\frac{z}{2}+\frac{\tau}{4} \\[1mm]
z_2 &=&\frac{z}{2}-\frac{\tau}{4}
\end{array}\right. $ respectively as follows:
\begin{subequations}
{\allowdisplaybreaks
\begin{eqnarray}
& & \hspace{-8mm}
\theta_{0, m+1}
\Big(\tau, \, - \, \dfrac12+\dfrac{m\tau}{2(m+1)}\Big) \, 
\Phi^{[\frac{m}{2}, \frac12]}
\Big(2\tau, \,\ z+\dfrac{\tau}{2}-\dfrac12, \,\ 
z-\dfrac{\tau}{2}+\dfrac12, \,\ 0\Big)
\nonumber
\\[2mm]
&=&
-i \, \eta(2\tau)^3 \, \Bigg\{ \, 
\frac{\displaystyle 
\theta_{0, m+1}\Big(\tau, \, \frac{\tau}{2(m+1)}+z\Big)
}{\vartheta_{10}(2\tau, \, z+\frac{\tau}{2})} 
\, - \, 
\frac{\displaystyle 
\theta_{0, m+1}\Big(\tau, \, \frac{\tau}{2(m+1)}-z\Big)
}{\vartheta_{10}(2\tau, \, z-\frac{\tau}{2})} \Bigg\}
\nonumber
\\[2mm]
& & \hspace{-5mm}
+ \, \sum_{j=1}^{\infty} \sum_{r=1}^j 
\sum_{\substack{k=1 \\[1mm] k \, : \, {\rm odd}}}^{m-1}
(-1)^j \, q^{(m+1)j^2-\frac{1}{4m}(2mr-k)^2} %
\nonumber
\\[2mm]
& & \hspace{-5mm}
\times \big\{
q^{(j+\frac14) (2mr-k)}e^{\frac{\pi i}{2}(2mr+k)}
+
q^{(j-\frac14) (2mr-k)}e^{\frac{\pi i}{2}(2mr-k)}
\big\}
\big[\theta_{k,m}-\theta_{-k,m}](\tau, z)
\nonumber
\\[2mm]
& & \hspace{-5mm}
- \, \sum_{j=1}^{\infty} \sum_{r=0}^{j-1} 
\sum_{\substack{k=1 \\[1mm] k \, : \, {\rm odd}}}^{m-1}
(-1)^j \, q^{(m+1)j^2-\frac{1}{4m}(2mr+k)^2} 
\nonumber
\\[2mm]
& & \hspace{-5mm}
\times 
\big\{
q^{(j+\frac14) (2mr+k)}e^{\frac{\pi i}{2}(2mr-k)}
+ 
q^{(j-\frac14) (2mr+k)}e^{\frac{\pi i}{2}(2mr+k)}
\big\}
\big[\theta_{k,m}-\theta_{-k,m}](\tau, z)
\nonumber
\\[0mm]
& &
\label{n3-2:eqn:2022-321a1}
\end{eqnarray}}
and
{\allowdisplaybreaks
\begin{eqnarray}
& & \hspace{-10mm}
\theta_{0, m+1}
\Big(\tau, \, - \, \dfrac{\tau+1}{2(m+1)}+\dfrac{\tau}{2}\Big) \, 
\Phi^{[\frac{m}{2}, \frac12]}
\Big(2\tau, \,\ z+\dfrac{\tau}{2}, \,\ 
z-\dfrac{\tau}{2}, \,\ 0\Big)
\nonumber
\\[2mm]
&=&
-i \, \eta(2\tau)^3 \, \Bigg\{ \, 
\frac{\displaystyle 
\theta_{0, m+1}\Big(\tau, \, 
\frac{\tau+1}{2(m+1)}+ z\Big)
}{\vartheta_{11}(2\tau, z+\frac{\tau}{2})} 
\, + \, 
\frac{\displaystyle 
\theta_{0, m+1}\Big(\tau, \, \frac{\tau+1}{2(m+1)}
- z\Big)
}{\vartheta_{11}(2\tau, z-\frac{\tau}{2})} \Bigg\}
\nonumber
\\[2mm]
& & \hspace{-5mm}
+ \, \sum_{j=1}^{\infty} \sum_{r=1}^j 
\sum_{\substack{k=1 \\[1mm] k \, : \, {\rm odd}}}^{m-1}
(-1)^j \, q^{(m+1)j^2-\frac{1}{4m}(2mr-k)^2}
\nonumber
\\[2mm]
& & 
\times \, \big\{q^{(j+\frac14)(2mr-k)}
\, + \, q^{(j-\frac14)(2mr-k)}
\big\}
\big[\theta_{k,m}-\theta_{-k,m}](\tau, z)
\nonumber
\\[2mm]
& & \hspace{-5mm}
- \, \sum_{j=1}^{\infty} \sum_{r=0}^{j-1} 
\sum_{\substack{k=1 \\[1mm] k \, : \, {\rm odd}}}^{m-1}
(-1)^j \, q^{(m+1)j^2-\frac{1}{4m}(2mr+k)^2}
\nonumber
\\[2mm]
& & 
\times \, \big\{
q^{(j+\frac14)(2mr+k)} \, + \, q^{(j-\frac14)(2mr+k)}
\big\}
\big[\theta_{k,m}-\theta_{-k,m}](\tau, z)
\label{n3-2:eqn:2022-321a2}
\end{eqnarray}}
\end{subequations}
Then by \eqref{n3-2:eqn:2022-321c} and \eqref{n3-2:eqn:2022-321a1}, we have 
\begin{subequations}
{\allowdisplaybreaks
\begin{eqnarray}
& & \hspace{-8mm}
\theta_{0, m+1}
\Big(\tau, \, - \, \dfrac12+\dfrac{m\tau}{2(m+1)}\Big) \, 
\big(\overset{N=3}{R}{}^{(+)} {\rm ch}_{H(\Lambda^{[K(m),0]})}\big)(\tau, z)
\nonumber
\\[2mm]
&=& 
q^{-\frac{m}{16}}
\theta_{0, m+1}
\Big(\tau, \, - \, \dfrac12+\dfrac{m\tau}{2(m+1)}\Big) \, 
\Phi^{[\frac{m}{2}, \frac12]}
\Big(2\tau, \,\ z+\dfrac{\tau}{2}-\dfrac12, \,\ 
z-\dfrac{\tau}{2}+\dfrac12, \,\ 0\Big)
\nonumber
\\[2mm]
&=&
-i q^{-\frac{m}{16}} \eta(2\tau)^3 \Bigg\{
\frac{\displaystyle 
\theta_{0, m+1}\Big(\tau, \frac{\tau}{2(m+1)}+z\Big)
}{\vartheta_{10}(2\tau, z+\frac{\tau}{2})} 
- 
\frac{\displaystyle 
\theta_{0, m+1}\Big(\tau, \frac{\tau}{2(m+1)}-z\Big)
}{\vartheta_{10}(2\tau, z-\frac{\tau}{2})} \Bigg\}
\nonumber
\\[2mm]
& & \hspace{-5mm}
+ \, \sum_{j=1}^{\infty} \sum_{r=1}^j 
\sum_{\substack{k=1 \\[1mm] k \, : \, {\rm odd}}}^{m-1}
(-1)^j \, q^{(m+1)j^2-\frac{1}{4m}(2mr-k)^2-\frac{m}{16}} %
\nonumber
\\[2mm]
& & \hspace{-5mm}
\times \big\{
q^{(j+\frac14) (2mr-k)}e^{\frac{\pi i}{2}(2mr+k)}
+
q^{(j-\frac14) (2mr-k)}e^{\frac{\pi i}{2}(2mr-k)}
\big\}
\big[\theta_{k,m}-\theta_{-k,m}](\tau, z)
\nonumber
\\[2mm]
& & \hspace{-5mm}
- \, \sum_{j=1}^{\infty} \sum_{r=0}^{j-1} 
\sum_{\substack{k=1 \\[1mm] k \, : \, {\rm odd}}}^{m-1}
(-1)^j \, q^{(m+1)j^2-\frac{1}{4m}(2mr+k)^2-\frac{m}{16}} 
\nonumber
\\[2mm]
& & \hspace{-5mm}
\times 
\big\{
q^{(j+\frac14) (2mr+k)}e^{\frac{\pi i}{2}(2mr-k)}
+ 
q^{(j-\frac14) (2mr+k)}e^{\frac{\pi i}{2}(2mr+k)}
\big\}
\big[\theta_{k,m}-\theta_{-k,m}](\tau, z)
\nonumber
\\[0mm]
& &
\label{n3-2:eqn:2022-323a1}
\end{eqnarray}}
and, by \eqref{n3-2:eqn:2022-321c} and 
\eqref{n3-2:eqn:2022-321a2}, we have
{\allowdisplaybreaks
\begin{eqnarray}
& & \hspace{-8mm}
\theta_{0, m+1}
\Big(\tau, \, - \, \dfrac{\tau+1}{2(m+1)}+\dfrac{\tau}{2}\Big) \, 
\big(\overset{N=3}{R}{}^{(-)} 
{\rm ch}_{H(\Lambda^{[K(m),0]})}^{(-)}\big)(\tau, z)
\nonumber
\\[2mm]
&= & 
q^{-\frac{m}{16}}
\theta_{0, m+1}
\Big(\tau, \, 
- \, \dfrac{\tau+1}{2(m+1)}+\dfrac{\tau}{2}
\Big) \, 
\Phi^{[\frac{m}{2}, \frac12]}
\Big(2\tau, \,\ z+\dfrac{\tau}{2}, \,\ 
z-\dfrac{\tau}{2}, \,\ 0\Big)
\nonumber
\\[2mm]
&=&
-i q^{-\frac{m}{16}} \eta(2\tau)^3 \Bigg\{ 
\frac{\displaystyle 
\theta_{0, m+1}\Big(\tau, 
\frac{\tau+1}{2(m+1)}+ z\Big)
}{\vartheta_{11}(2\tau, z+\frac{\tau}{2})} 
+ 
\frac{\displaystyle 
\theta_{0, m+1}\Big(\tau, \frac{\tau+1}{2(m+1)}
- z\Big)
}{\vartheta_{11}(2\tau, z-\frac{\tau}{2})} \Bigg\}
\nonumber
\\[2mm]
& & \hspace{-5mm}
+ \, \sum_{j=1}^{\infty} \sum_{r=1}^j 
\sum_{\substack{k=1 \\[1mm] k \, : \, {\rm odd}}}^{m-1}
(-1)^j \, q^{(m+1)j^2-\frac{1}{4m}(2mr-k)^2-\frac{m}{16}}
\nonumber
\\[2mm]
& & 
\times \, \big\{q^{(j+\frac14)(2mr-k)}
\, + \, q^{(j-\frac14)(2mr-k)}
\big\}
\big[\theta_{k,m}-\theta_{-k,m}](\tau, z)
\nonumber
\\[2mm]
& & \hspace{-5mm}
- \, \sum_{j=1}^{\infty} \sum_{r=0}^{j-1} 
\sum_{\substack{k=1 \\[1mm] k \, : \, {\rm odd}}}^{m-1}
(-1)^j \, q^{(m+1)j^2-\frac{1}{4m}(2mr+k)^2-\frac{m}{16}}
\nonumber
\\[2mm]
& & 
\times \, \big\{
q^{(j+\frac14)(2mr+k)} \, + \, q^{(j-\frac14)(2mr+k)}
\big\}
\big[\theta_{k,m}-\theta_{-k,m}](\tau, z)
\label{n3-2:eqn:2022-323a2}
\end{eqnarray}}
\end{subequations}
Then, dividing both sides of \eqref{n3-2:eqn:2022-323a1} and 
\eqref{n3-2:eqn:2022-323a2} by \\
$\theta_{0, m+1}(\tau, -\frac12+\frac{m\tau}{2(m+1)}) 
\overset{N=3}{R}{}^{(+)}(\tau,z)$
and 
$\theta_{0, m+1}(\tau, - \frac{\tau+1}{2(m+1)}+\frac{\tau}{2}) 
\overset{N=3}{R}{}^{(-)}(\tau,z)$ \\
respectively, we obtain the following formulas:
\begin{subequations}
{\allowdisplaybreaks
\begin{eqnarray}
& & \hspace{-5mm}
{\rm ch}_{H(\Lambda^{[K(m), 0]})}^{(+)}(\tau,z)
\nonumber
\\[2mm]
& & \hspace{-10mm}
= \,\ -i \, q^{-\frac{m}{16}}
\frac{\eta(2\tau)^2}{\eta(\frac{\tau}{2})} \cdot
\frac{\vartheta_{00}(\tau, z)}{\vartheta_{11}(\tau, z)} \cdot
\frac{1}{\theta_{0, m+1}\Big(\tau, \, - \, \dfrac12+\dfrac{m\tau}{2(m+1)}\Big)}
\nonumber
\\[2mm]
& & \hspace{5mm}
\times \, 
\Bigg\{ \, 
\frac{\displaystyle 
\theta_{0, m+1}\Big(\tau, \, \frac{\tau}{2(m+1)}+z\Big)
}{\vartheta_{10}(2\tau, \, z+\frac{\tau}{2})} 
\, - \, 
\frac{\displaystyle 
\theta_{0, m+1}\Big(\tau, \, \frac{\tau}{2(m+1)}-z\Big)
}{\vartheta_{10}(2\tau, \, z-\frac{\tau}{2})} \Bigg\}
\nonumber
\\[3mm]
& & \hspace{-10mm}
+ \,\ 
\frac{1}{\eta(\frac{\tau}{2})\eta(2\tau)} \cdot
\frac{\vartheta_{00}(\tau, z)}{\vartheta_{11}(\tau, z)} \cdot
\frac{1}{\theta_{0, m+1}\Big(\tau, \, - \, \dfrac12+\dfrac{m\tau}{2(m+1)}\Big)}
\nonumber
\\[2mm]
& &\hspace{-10mm}
\times \, \Bigg\{
\sum_{j=1}^{\infty} \sum_{r=1}^j 
\sum_{\substack{k=1 \\[1mm] k \, : \, {\rm odd}}}^{m-1}
(-1)^j \, q^{(m+1)j^2-\frac{1}{4m}(2mr-k)^2-\frac{m}{16}} 
\nonumber
\\[2mm]
& & \hspace{-10mm}
\times \big[
q^{(j+\frac14) (2mr-k)}e^{\frac{\pi i}{2}(2mr+k)}
+
q^{(j-\frac14) (2mr-k)}e^{\frac{\pi i}{2}(2mr-k)}
\big]
\big[\theta_{k,m}-\theta_{-k,m}](\tau, z)
\nonumber
\\[2mm]
& & \hspace{-10mm}
- \, \sum_{j=1}^{\infty} \sum_{r=0}^{j-1} 
\sum_{\substack{k=1 \\[1mm] k \, : \, {\rm odd}}}^{m-1}
(-1)^j \, q^{(m+1)j^2-\frac{1}{4m}(2mr+k)^2-\frac{m}{16}} 
\nonumber
\\[0mm]
& & \hspace{-10mm}
\times 
\big[
q^{(j+\frac14) (2mr+k)}e^{\frac{\pi i}{2}(2mr-k)}
+ 
q^{(j-\frac14) (2mr+k)}e^{\frac{\pi i}{2}(2mr+k)}
\big]
\big[\theta_{k,m}-\theta_{-k,m}](\tau, z)
\Bigg\}
\nonumber
\\[0mm]
& &
\label{n3-2:eqn:2022-331a1}
\end{eqnarray}}

\vspace{-7mm}

\noindent
and
{\allowdisplaybreaks
\begin{eqnarray}
& & \hspace{-10mm}
{\rm ch}_{H(\Lambda^{[K(m), 0]})}^{(-)}(\tau,z)
\nonumber
\\[2mm]
&=&
-i \, q^{-\frac{m}{16}} \, 
\frac{\eta(\frac{\tau}{2})\eta(2\tau)^3}{\eta(\tau)^3} \cdot
\frac{\vartheta_{01}(\tau,z)}{\vartheta_{11}(\tau,z)} \cdot
\frac{1}{\theta_{0, m+1}
\Big(\tau, \, - \, \dfrac{\tau+1}{2(m+1)}+\dfrac{\tau}{2}\Big)}
\nonumber
\\[2mm]
& &
\times 
\Bigg\{ \, 
\frac{\displaystyle 
\theta_{0, m+1}\Big(\tau, \, 
\frac{\tau+1}{2(m+1)}+ z\Big)
}{\vartheta_{11}(2\tau, z+\frac{\tau}{2})} 
\, + \, 
\frac{\displaystyle 
\theta_{0, m+1}\Big(\tau, \, \frac{\tau+1}{2(m+1)}
- z\Big)
}{\vartheta_{11}(2\tau, z-\frac{\tau}{2})} 
\Bigg\}
\nonumber
\\[2mm]
& & \hspace{-5mm}
+ \, 
\frac{\eta(\frac{\tau}{2})}{\eta(\tau)^3} \cdot
\frac{\vartheta_{01}(\tau,z)}{\vartheta_{11}(\tau,z)} \cdot
\frac{1}{\theta_{0, m+1}
\Big(\tau, \, - \, \dfrac{\tau+1}{2(m+1)}+\dfrac{\tau}{2}\Big)}
\nonumber
\\[2mm]
& &
\times \Bigg\{
\sum_{j=1}^{\infty} \sum_{r=1}^j 
\sum_{\substack{k=1 \\[1mm] k \, : \, {\rm odd}}}^{m-1}
(-1)^j \, q^{(m+1)j^2-\frac{1}{4m}(2mr-k)^2-\frac{m}{16}}
\nonumber
\\[2mm]
& & 
\times \, \big\{q^{(j+\frac14)(2mr-k)}
\, + \, q^{(j-\frac14)(2mr-k)}
\big\}
\big[\theta_{k,m}-\theta_{-k,m}](\tau, z)
\nonumber
\\[2mm]
& & \hspace{-5mm}
- \, \sum_{j=1}^{\infty} \sum_{r=0}^{j-1} 
\sum_{\substack{k=1 \\[1mm] k \, : \, {\rm odd}}}^{m-1}
(-1)^j \, q^{(m+1)j^2-\frac{1}{4m}(2mr+k)^2-\frac{m}{16}}
\nonumber
\\[0mm]
& & 
\times \, \big\{
q^{(j+\frac14)(2mr+k)} \, + \, q^{(j-\frac14)(2mr+k)}
\big\}
\big[\theta_{k,m}-\theta_{-k,m}](\tau, z)
\Bigg\}
\label{n3-2:eqn:2022-331a2}
\end{eqnarray}}
\end{subequations}

Rewriting \eqref{n3-2:eqn:2022-331a1} and \eqref{n3-2:eqn:2022-331a2}
by using the formulas 
\begin{subequations}
{\allowdisplaybreaks
\begin{eqnarray}
& &\left\{
\begin{array}{ccr}
\vartheta_{10}(2\tau, z\pm \frac{\tau}{2}) &=&
q^{-\frac{1}{16}} \, e^{\mp \frac{\pi iz}{2}} \, 
\theta_{\mp \frac12, 1}(\tau, z)
\\[2mm]
\vartheta_{11}(2\tau, z\pm \frac{\tau}{2}) &=&
\mp \, i \, q^{-\frac{1}{16}} \, e^{\mp \frac{\pi iz}{2}} \, 
\theta^{(-)}_{\mp \frac12, 1}(\tau, z) 
\end{array}\right.
\label{n3-2:eqn:2022-331b1}
\\[2mm]
& &\left\{
\begin{array}{lcl}
\theta_{0.m+1}(\tau, -\frac12+\frac{m\tau}{2(m+1)}) &=&
q^{-\frac{m^2}{16(m+1)}} \, e^{\frac{\pi im}{4}} \, 
\theta_{-\frac{m}{2}, m+1}(\tau, \frac12)
\\[2mm]
\theta_{0.m+1}(\tau, -\frac{\tau+1}{2(m+1)}+\frac{\tau}{2}) 
&=&
q^{-\frac{m^2}{16(m+1)}} \, 
\theta_{\frac{m}{2}, m+1}^{(-)}(\tau, 0)
\end{array}\right. \hspace{10mm}
\label{n3-2:eqn:2022-331b2}
\\[2mm]
& &\left\{
\begin{array}{lcl}
\theta_{0.m+1}(\tau, \frac{\tau}{2(m+1)} \pm z) &=&
q^{-\frac{1}{16(m+1)}} \, e^{\mp \frac{\pi iz}{2}} \, 
\theta_{\pm \frac12, m+1}(\tau,z)
\\[2mm]
\theta_{0.m+1}(\tau, \frac{\tau+1}{2(m+1)} \pm z) &=&
q^{-\frac{1}{16(m+1)}} \, e^{\mp \frac{\pi iz}{2}} \, 
\theta_{\pm \frac12, m+1}^{(-)}(\tau,z)
\end{array}\right.
\label{n3-2:eqn:2022-331b3}
\end{eqnarray}}
\end{subequations}
we obtain Proposition \ref{n3-2:prop:2022-320a}.
\end{proof}

\begin{prop} 
\label{n3-2:prop:2022-331b}
For $m \in \nnn_{\rm odd}$, the character and the supercharacter of 
the N=3 module $H(\Lambda^{[K(m), 1]})$ are given by the following 
formulas:

\begin{enumerate}
\item[{\rm 1)}] \quad ${\rm ch}_{H(\Lambda^{[K(m), 1]})}^{(+)}(\tau, z)$

\vspace{-2mm}

{\allowdisplaybreaks
\begin{eqnarray*}
& & \hspace{-10mm}
= \,\ - \, i \,\ 
\frac{\eta(2\tau)^2}{\eta(\frac{\tau}{2})} \cdot 
\frac{\vartheta_{00}(\tau,z)}{\vartheta_{11}(\tau,z)} \cdot 
\frac{1}{\theta_{\frac{m}{2},m+1}(\tau, 0)} 
\\[2mm]
& & \hspace{15mm}
\times \,\ \Bigg\{ 
\frac{\theta_{m+\frac12, m+1}(\tau,z)}{\theta_{-\frac12,1}(\tau,z)}
\,\ - \,\ 
\frac{\theta_{-m-\frac12, m+1}(\tau,z)}{\theta_{\frac12,1}(\tau,z)}
\Bigg\}
\\[3mm]
& & \hspace{-7mm}
- \,\ 
q^{-\frac{m}{16(m+1)}} \, e^{-\frac{\pi im}{2}} \, 
\frac{1}{\eta(\frac{\tau}{2}) \eta(2\tau)} \cdot 
\frac{\vartheta_{00}(\tau,z)}{\vartheta_{11}(\tau,z)} \cdot 
\frac{1}{\theta_{\frac{m}{2},m+1}(\tau, 0)} 
\\[3mm]
& & \hspace{-5mm}
\times \,\ \Bigg\{
\sum_{j=1}^{\infty} \sum_{r=1}^j 
\sum_{\substack{k=0 \\[1mm] k \, : \, {\rm even}}}^{m-1}
(-1)^j q^{(m+1)j^2-\frac{1}{4m}(2m(j-r)+m-k)^2} 
\\[2mm]
& & 
\times \,\ \big\{ 
e^{-\frac{\pi i}{2}(2mr-m+k)} q^{-\frac14(2mr-m+k)}
\, + \, 
e^{\frac{\pi i}{2}(2mr-m+k)} q^{\frac14(2mr-m+k)}
\big\}
\\[2mm]
& &
\times \,\ \big[\theta_{k,m}- \theta_{-k,m}](\tau, z)
\\[3mm]
& & \hspace{-5mm}
+ \,\ 
\sum_{j=1}^{\infty} \sum_{r=0}^{j-1} 
\sum_{\substack{k=0 \\[1mm] k \, : \, {\rm even}}}^{m-1}
(-1)^j q^{(m+1)j^2-\frac{1}{4m}(2m(j-r)-m+k)^2} 
\\[0mm]
& & 
\times \,\ \big\{ 
e^{-\frac{\pi i}{2}(2mr+m-k)} q^{-\frac14(2mr+m-k)}
\, + \, 
e^{\frac{\pi i}{2}(2mr+m-k)} q^{\frac14(2mr+m-k)}
\big\}
\\[2mm]
& &
\times \,\ \big[\theta_{k,m}- \theta_{-k,m}](\tau, z)
\Bigg\}
\\[1mm]
& & \hspace{-5mm}
+ \,\ \frac{1}{\eta(\frac{\tau}{2}) \eta(2\tau)} \cdot 
\frac{\vartheta_{00}(\tau,z)}{\vartheta_{11}(\tau,z)} 
\sum_{k=1}^{\frac{m-1}{2}}
(-1)^k \, q^{-\frac{1}{16m}(4k-m)^2} \, 
\big[\theta_{k,\frac{m}{2}}-\theta_{-k,\frac{m}{2}}\big](2\tau, 2z)
\end{eqnarray*}}

\item[{\rm 2)}] \quad ${\rm ch}_{H(\Lambda^{[K(m), 1]})}^{(-)}(\tau, z)$

\vspace{-3mm}

{\allowdisplaybreaks
\begin{eqnarray*}
& & \hspace{-10mm}
= \,\ 
\frac{\eta(\frac{\tau}{2})\eta(2\tau)^3}{\eta(\tau)^3} \cdot 
\frac{\vartheta_{00}(\tau,z)}{\vartheta_{11}(\tau,z)} \cdot 
\frac{1}{\theta_{\frac{m}{2},m+1}^{(-)}(\tau, 0)} 
\\[1mm]
& & \hspace{15mm}
\times \,\ \Bigg\{ 
\frac{\theta_{m+\frac12, m+1}^{(-)}(\tau,z)
}{\theta_{-\frac12, 1}^{(-)}(\tau,z)}
\,\ - \,\ 
\frac{\theta_{-m-\frac12, m+1}^{(-)}(\tau,z)
}{\theta_{\frac12, 1}^{(-)}(\tau,z)} \Bigg\}
\nonumber
\\[3mm]
& & \hspace{-8mm}
- \,\ 
q^{-\frac{m}{16(m+1)}} \, \frac{\eta(\frac{\tau}{2})}{\eta(\tau)^3} \cdot 
\frac{\vartheta_{00}(\tau,z)}{\vartheta_{11}(\tau,z)} \cdot 
\frac{1}{\theta_{\frac{m}{2},m+1}^{(-)}(\tau, 0)} \, 
\\[3mm]
& & \hspace{-5mm}
\times \,\ \Bigg\{
\sum_{j=1}^{\infty} \sum_{r=1}^j 
\sum_{\substack{k=0 \\[1mm] k \, : \, {\rm even}}}^{m-1}
(-1)^j q^{(m+1)j^2-\frac{1}{4m}(2m(j-r)+m-k)^2} 
\nonumber
\\[2mm]
& & \hspace{5mm}
\times \,\ \big\{ 
q^{-\frac14(2mr-m+k)} \, + \, q^{\frac14(2mr-m+k)}\big\} \, 
\big[\theta_{k,m}- \, \theta_{-k,m}](\tau, z)
\nonumber
\\[3mm]
& & \hspace{-2mm}
+ \,\ 
\sum_{j=1}^{\infty} \sum_{r=0}^{j-1} 
\sum_{\substack{k=0 \\[1mm] k \, : \, {\rm even}}}^{m-1}
(-1)^j q^{(m+1)j^2-\frac{1}{4m}(2m(j-r)-m+k)^2} 
\nonumber
\\[0mm]
& & \hspace{5mm}
\times \,\ \big\{
q^{-\frac14(2mr+m-k)} \, + \, q^{\frac14(2mr+m-k)}\big\} 
\, \big[\theta_{k,m}- \, \theta_{-k,m}](\tau, z)
\Bigg\}
\\[0mm]
& & \hspace{-7mm}
+ \,\ 
\frac{\eta(\frac{\tau}{2})}{\eta(\tau)^3} \cdot 
\frac{\vartheta_{00}(\tau,z)}{\vartheta_{11}(\tau,z)}  \, 
\sum_{k=1}^{\frac{m-1}{2}} \,\ q^{-\frac{1}{16m}(4k-m)^2}
\,\ 
\big[\theta_{k,\frac{m}{2}}-\theta_{-k,\frac{m}{2}}\big](2\tau, 2z)
\end{eqnarray*}}
\end{enumerate}
\end{prop}

\begin{proof} Letting $m_2=1$ in \eqref{n3-2:eqn:2022-321b1} and 
using \eqref{n3-2:eqn:2022-331d}, we have
{\allowdisplaybreaks
\begin{eqnarray}
\big(\overset{N=3}{R}{}^{(+)}{\rm ch}_{H(\Lambda^{[K(m), 1]})}^{(+)}\big)
(\tau,z) \hspace{-2mm} &=& \hspace{-2mm}
q^{-\frac{m}{16}} \, 
\Phi^{[\frac{m}{2}, 0]}
\Big(2\tau, z+\frac{\tau}{2}-\frac12, z-\frac{\tau}{2}+\frac12, 0\Big)
\nonumber
\\[2mm]
\big(\overset{N=3}{R}{}^{(-)}{\rm ch}_{H(\Lambda^{[K(m), 1]})}^{(-)}\big)
(\tau,z) \hspace{-2mm} &=& \hspace{-2mm}
q^{-\frac{m}{16}} \, 
\Phi^{[\frac{m}{2}, 0]}
\Big(2\tau, z+\frac{\tau}{2}, z-\frac{\tau}{2}, 0\Big) \, .
\label{n3-2:eqn:2022-401a}
\end{eqnarray}}

\vspace{-2mm}

\noindent
The functions $\Phi^{[\frac{m}{2}, 0]}
(2\tau, z+\frac{\tau}{2}-\frac12, z-\frac{\tau}{2}+\frac12, 0)$
and 
$\Phi^{[\frac{m}{2}, 0]}
(2\tau, z+\frac{\tau}{2}, z-\frac{\tau}{2}, 0)$ 
in the RHS of \eqref{n3-2:eqn:2022-401a}
are obtained from Proposition \ref{n3-2:prop:2022-331a} by 
replacing $m$ with $\frac{m}{2}$ and putting $\left\{
\begin{array}{ccc}
z_1 &=&\frac{z}{2}+\frac{\tau}{4}-\frac14 \\[1mm]
z_2 &=&\frac{z}{2}-\frac{\tau}{4}+\frac14
\end{array}\right. $ and $\left\{
\begin{array}{ccc}
z_1 &=&\frac{z}{2}+\frac{\tau}{4} \\[1mm]
z_2 &=&\frac{z}{2}-\frac{\tau}{4}
\end{array}\right. $ respectively, and are calculated 
by using \eqref{n3-2:eqn:2022-331b1} and the formulas
{\allowdisplaybreaks
\begin{eqnarray*}
& &\left\{
\begin{array}{lcl}
\theta_{-m,m+1}^{(-)}(\tau, \frac{m(\tau-1)}{2(m+1)}) &=&
q^{-\frac{m^2}{16(m+1)}} e^{\frac{\pi im^2}{2(m+1)}} 
\theta_{\frac{m}{2}, m+1}(\tau,0) \quad (m \in \nnn_{\rm odd})
\\[2mm]
\theta_{-m,m+1}^{(-)}(\tau, \frac{m\tau}{2(m+1)}) &=&
q^{-\frac{m^2}{16(m+1)}} \, 
\theta_{\frac{m}{2}, m+1}^{(-)}(\tau,0)
\end{array}\right.
\\[2mm]
& & \hspace{3mm}
\theta_{m,m+1}^{(-)}(\tau, \frac{\tau-1}{2(m+1)}\pm z) 
\, = \, q^{-\frac{1}{16(m+1)}} \, 
e^{-\frac{\pi im}{2(m+1)}} \, 
e^{\mp \frac{\pi iz}{2}} \, 
\theta_{\pm (m+\frac12),m+1}(\tau, z) \hspace{10mm}
\\[2mm]
& & \hspace{3mm}
\theta_{m,m+1}^{(-)}(\tau, \frac{\tau}{2(m+1)}\pm z) 
\, = \, q^{-\frac{1}{16(m+1)}} \, 
e^{\mp \frac{\pi iz}{2}} \, 
\theta_{\pm (m+\frac12),m+1}^{(-)}(\tau, z)
\end{eqnarray*}}
as follows:  %
\begin{subequations}
{\allowdisplaybreaks
\begin{eqnarray}
& & \hspace{-5mm}
\theta_{\frac{m}{2},m+1}(\tau, 0) \, 
\Phi^{[\frac{m}{2}, 0]}\Big(2\tau, \,\ z+ \dfrac{\tau}{2}-\dfrac12, \,\ 
z- \dfrac{\tau}{2}+\dfrac12, \,\ \dfrac{\tau}{8}\Big)
\nonumber
\\[2mm]
& & \hspace{-10mm}
= \,\ - \, i \, 
\eta(2\tau)^3 \Bigg\{ 
\frac{\theta_{m+\frac12, m+1}(\tau,z)}{\theta_{-\frac12,1}(\tau,z)}
\,\ - \,\ 
\frac{\theta_{-m-\frac12, m+1}(\tau,z)}{\theta_{\frac12,1}(\tau,z)}
\Bigg\}
\nonumber
\\[2mm]
& & \hspace{-5mm}
- \,\ 
q^{-\frac{m}{16(m+1)}} \, e^{-\frac{\pi im}{2}} \, 
\sum_{j=1}^{\infty} \sum_{r=1}^j 
\sum_{\substack{k=0 \\[1mm] k \, : \, {\rm even}}}^{m-1}
(-1)^j q^{(m+1)j^2-\frac{1}{4m}(2m(j-r)+m-k)^2} 
\nonumber
\\[2mm]
& & 
\times \, \big\{ 
e^{-\frac{\pi i}{2}(2mr-m+k)} q^{-\frac14(2mr-m+k)}
\, + \, 
e^{\frac{\pi i}{2}(2mr-m+k)} q^{\frac14(2mr-m+k)}
\big\}
\nonumber
\\[2mm]
& & 
\times \,\ \big[\theta_{k,m}- \theta_{-k,m}](\tau, z)
\nonumber
\\[2mm]
& & \hspace{-5mm}
+ \,\ 
q^{-\frac{m}{16(m+1)}} \, e^{-\frac{\pi im}{2}} \, 
\sum_{j=1}^{\infty} \sum_{r=0}^{j-1} 
\sum_{\substack{k=0 \\[1mm] k \, : \, {\rm even}}}^{m-1}
(-1)^j q^{(m+1)j^2-\frac{1}{4m}(2m(j-r)-m+k)^2} 
\nonumber
\\[2mm]
& & 
\times \, \big\{ 
e^{-\frac{\pi i}{2}(2mr+m-k)} q^{-\frac14(2mr+m-k)}
\, + \, 
e^{\frac{\pi i}{2}(2mr+m-k)} q^{\frac14(2mr+m-k)}
\big\}
\nonumber
\\[2mm]
& &
\times \,\ \big[\theta_{k,m}- \theta_{-k,m}](\tau, z)
\nonumber
\\[2mm]
& & \hspace{-5mm}
+ \,\ 
\theta_{\frac{m}{2},m+1}(\tau, 0) \, 
\sum_{k=1}^{\frac{m-1}{2}}
(-1)^k \, q^{-\frac{1}{16m}(4k-m)^2} \, 
\big[\theta_{k,\frac{m}{2}}-\theta_{-k,\frac{m}{2}}\big](2\tau, 2z)
\label{n3-2:eqn:2022-401b1}
\end{eqnarray}}
and  %
{\allowdisplaybreaks
\begin{eqnarray}
& & \hspace{-5mm}
\theta_{\frac{m}{2},m+1}^{(-)}(\tau, 0) \, 
\Phi^{[\frac{m}{2}, 0]}\Big(2\tau, \,\ z+\dfrac{\tau}{2}, \,\ 
z-\dfrac{\tau}{2}, \,\ \dfrac{\tau}{8}\Big)
\nonumber
\\[2mm]
& & \hspace{-10mm}
= \,\ 
\eta(2\tau)^3 \, \Bigg\{ 
\frac{\theta_{m+\frac12, m+1}^{(-)}(\tau,z)
}{\theta_{-\frac12, 1}^{(-)}(\tau,z)}
\,\ - \,\ 
\frac{\theta_{-m-\frac12, m+1}^{(-)}(\tau,z)
}{\theta_{\frac12, 1}^{(-)}(\tau,z)} \Bigg\} 
\nonumber
\\[3mm]
& & \hspace{-5mm}
- \,\ 
q^{\frac{-m}{16(m+1)}} \, 
\sum_{j=1}^{\infty} \sum_{r=1}^j 
\sum_{\substack{k=0 \\[1mm] k \, : \, {\rm even}}}^{m-1}
(-1)^j q^{(m+1)j^2-\frac{1}{4m}(2m(j-r)+m-k)^2} 
\nonumber
\\[2mm]
& & 
\times \,\ \big\{ 
q^{-\frac14(2mr-m+k)} \, + \, q^{\frac14(2mr-m+k)}\big\} \, 
\big[\theta_{k,m}- \, \theta_{-k,m}](\tau, z)
\nonumber
\\[2mm]
& & \hspace{-5mm}
+ \,\ 
q^{-\frac{m}{16(m+1)}} \, 
\sum_{j=1}^{\infty} \sum_{r=0}^{j-1} 
\sum_{\substack{k=0 \\[1mm] k \, : \, {\rm even}}}^{m-1}
(-1)^j q^{(m+1)j^2-\frac{1}{4m}(2m(j-r)-m+k)^2} 
\nonumber
\\[2mm]
& & 
\times \,\ \big\{
q^{-\frac14(2mr+m-k)} \, + \, q^{\frac14(2mr+m-k)}\big\} 
\, \big[\theta_{k,m}- \, \theta_{-k,m}](\tau, z)
\nonumber
\\[3mm]
& & \hspace{-5mm}
+ \,\ 
\theta_{\frac{m}{2},m+1}^{(-)}(\tau, 0) \, 
\sum_{k=1}^{\frac{m-1}{2}} \,\ q^{-\frac{1}{16m}(4k-m)^2}
\,\ 
\big[\theta_{k,\frac{m}{2}}-\theta_{-k,\frac{m}{2}}\big](2\tau, 2z)
\label{n3-2:eqn:2022-401b2}
\end{eqnarray}}
\end{subequations}
Substituting \eqref{n3-2:eqn:2022-401b1} and \eqref{n3-2:eqn:2022-401b2}
into \eqref{n3-2:eqn:2022-401a} and using the formula
\eqref{n3-2:eqn:2022-321b2} for $\overset{N=3}{R}{}^{(\pm)}(\tau,z)$, 
we obtain Proposition \ref{n3-2:prop:2022-331b}.
\end{proof}

We note that the character and the supercharacter of the N=3 module 
$H(\Lambda^{[K(m), m_2]})$ for all $m_2$ are obtained from those of 
$H(\Lambda^{[K(m), 0]})$ and $H(\Lambda^{[K(m), 1]})$
by the following:

\begin{prop}
\label{n3-2:prop:2022-320b}
Let $m \in \nnn$ and $m_2 \in \zzz_{\geq 0}$ such that 
$m_2 \leq m$. Then 
\begin{enumerate}
\item[{\rm 1)}] \,\ if \, $m_2 \in \zzz_{\rm even}$, 
\begin{enumerate}
\item[{\rm (i)}] ${\rm ch}_{H(\Lambda^{[K(m), m_2]})}^{(+)}(\tau,z) 
-
{\rm ch}_{H(\Lambda^{[K(m), 0]})}^{(+)}(\tau,z) 
\,\ = \,\ 
- \, i q^{-\frac{m}{16}}
\dfrac{1}{\eta(\frac{\tau}{2})\eta(2\tau)}$

$$ \hspace{-5mm}
\times \, 
\frac{\vartheta_{00}(\tau,z)}{\vartheta_{11}(\tau,z)}
\sum_{k=0}^{\frac{m_2}{2}-1}
(-1)^k \, q^{-\frac{1}{m}(k+\frac12)^2 + \frac12 (k+\frac12)} \, 
\big[\theta_{k+\frac12, \frac{m}{2}}-\theta_{-(k+\frac12), \frac{m}{2}}\big]
(2\tau, 2z)
$$

\item[{\rm (ii)}] ${\rm ch}_{H(\Lambda^{[K(m), m_2]})}^{(-)}(\tau,z) 
-
{\rm ch}_{H(\Lambda^{[K(m), 0]})}^{(-)}(\tau,z) 
\,\ = \,\ 
q^{-\frac{m}{16}}
\dfrac{\eta(\frac{\tau}{2})}{\eta(\tau)^3} \cdot 
\dfrac{\vartheta_{01}(\tau,z)}{\vartheta_{11}(\tau,z)}$
$$
\times \sum_{k=0}^{\frac{m_2}{2}-1}
q^{-\frac{1}{m}(k+\frac12)^2 + \frac12 (k+\frac12)} \, 
\big[\theta_{k+\frac12, \frac{m}{2}}-\theta_{-(k+\frac12), \frac{m}{2}}\big]
(2\tau, 2z)
$$
\end{enumerate}

\item[{\rm 2)}] \,\ if \, $m_2 \in \zzz_{\rm odd}$, 
\begin{enumerate}
\item[{\rm (i)}] ${\rm ch}_{H(\Lambda^{[K(m), m_2]})}^{(+)}(\tau,z) 
-
{\rm ch}_{H(\Lambda^{[K(m), 1]})}^{(+)}(\tau,z) 
\,\ = \,\ 
- \, i q^{-\frac{m}{16}}
\dfrac{1}{\eta(\frac{\tau}{2})\eta(2\tau)}$
$$ \hspace{-5mm}
\times \dfrac{\vartheta_{00}(\tau,z)}{\vartheta_{11}(\tau,z)}
\sum_{k=0}^{\frac{m_2-1}{2}}
(-1)^k \, q^{-\frac{k^2}{m} + \frac{k}{2}} \, 
\big[\theta_{k, \frac{m}{2}}-\theta_{-k, \frac{m}{2}}\big]
(2\tau, 2z)
$$

\item[{\rm (ii)}] ${\rm ch}_{H(\Lambda^{[K(m), m_2]})}^{(-)}(\tau,z) 
-
{\rm ch}_{H(\Lambda^{[K(m), 1]})}^{(-)}(\tau,z) 
\,\ = \,\ 
q^{-\frac{m}{16}}
\dfrac{\eta(\frac{\tau}{2})}{\eta(\tau)^3} \cdot 
\dfrac{\vartheta_{01}(\tau,z)}{\vartheta_{11}(\tau,z)}$
$$
\times \sum_{k=0}^{\frac{m_2-1}{2}}
q^{-\frac{k^2}{m} + \frac{k}{2}} \, 
\big[\theta_{k, \frac{m}{2}}-\theta_{-k, \frac{m}{2}}\big]
(2\tau, 2z)
$$
\end{enumerate}
\end{enumerate}
\end{prop}

\begin{proof} By Proposition 4.1 in \cite{W2022}, character and 
supercharacter of the N=3 module $H(\Lambda^{[K(m), m_2]})$ are given by
the formulas \eqref{n3-2:eqn:2022-321b1} and \eqref{n3-2:eqn:2022-321b2}.
Then, in the case $m_2 \in 2\zzz$, we have by Lemma \ref{n3-2:lemma:2022-320b} 
the following:
{\allowdisplaybreaks
\begin{eqnarray*}
& & \hspace{-10mm}
\big(\overset{N=3}{R}{}^{(+)}{\rm ch}_{H(\Lambda^{[K(m), m_2]})}^{(+)}\big)
(\tau,z) 
-
\big(\overset{N=3}{R}{}^{(+)}{\rm ch}_{H(\Lambda^{[K(m), 0]})}^{(+)}\big)
(\tau,z) 
\\[2mm]
&=&
\Phi^{[\frac{m}{2}, \frac{m_2+1}{2}]}
\Big(2\tau, z+\frac{\tau}{2}-\frac12, z-\frac{\tau}{2}+\frac12, \frac{\tau}{8}\Big)
\\[2mm]
& &
- \,\
\Phi^{[\frac{m}{2}, \frac12]}
\Big(2\tau, z+\frac{\tau}{2}-\frac12, z-\frac{\tau}{2}+\frac12, \frac{\tau}{8}\Big)
\\[2mm]
&=&
- \, i q^{-\frac{m}{16}}\sum_{k=0}^{\frac{m_2}{2}-1}
(-1)^k \, q^{-\frac{1}{m}(k+\frac12)^2 + \frac12 (k+\frac12)} \, 
\big[\theta_{k+\frac12, \frac{m}{2}}-\theta_{-(k+\frac12), \frac{m}{2}}\big]
(2\tau, 2z)
\\[2mm]
& & \hspace{-10mm}
\big(\overset{N=3}{R}{}^{(-)}{\rm ch}_{H(\Lambda^{[K(m), m_2]})}^{(-)}\big)
(\tau,z) 
-
\big(\overset{N=3}{R}{}^{(-)}{\rm ch}_{H(\Lambda^{[K(m), 0]})}^{(-)}\big)
(\tau,z) 
\\[2mm]
&=&
\Phi^{[\frac{m}{2}, \frac{m_2+1}{2}]}
\Big(2\tau, z+\frac{\tau}{2}, z-\frac{\tau}{2}, \frac{\tau}{8}\Big)
- 
\Phi^{[\frac{m}{2}, \frac12]}
\Big(2\tau, z+\frac{\tau}{2}, z-\frac{\tau}{2}, \frac{\tau}{8}\Big)
\\[2mm]
&=&
q^{-\frac{m}{16}}\sum_{k=0}^{\frac{m_2}{2}-1}
q^{-\frac{1}{m}(k+\frac12)^2 + \frac12 (k+\frac12)} \, 
\big[\theta_{k+\frac12, \frac{m}{2}}-\theta_{-(k+\frac12), \frac{m}{2}}\big]
(2\tau, 2z)
\end{eqnarray*}}
so
{\allowdisplaybreaks
\begin{eqnarray*}
& & \hspace{-10mm}
{\rm ch}_{H(\Lambda^{[K(m), m_2]})}^{(+)}(\tau,z) 
-
{\rm ch}_{H(\Lambda^{[K(m), 0]})}^{(+)}(\tau,z) 
\, = \, - \, i q^{-\frac{m}{16}}
\frac{1}{\eta(\frac{\tau}{2})\eta(2\tau)} \cdot 
\frac{\vartheta_{00}(\tau,z)}{\vartheta_{11}(\tau,z)}
\\[2mm]
& &
\times \sum_{k=0}^{\frac{m_2}{2}-1}
(-1)^k \, q^{-\frac{1}{m}(k+\frac12)^2 + \frac12 (k+\frac12)} \, 
\big[\theta_{k+\frac12, \frac{m}{2}}-\theta_{-(k+\frac12), \frac{m}{2}}\big]
(2\tau, 2z)
\\[2mm]
& & \hspace{-10mm}
{\rm ch}_{H(\Lambda^{[K(m), m_2]})}^{(-)}(\tau,z) 
-
{\rm ch}_{H(\Lambda^{[K(m), 0]})}^{(-)}(\tau,z) 
\, = \, q^{-\frac{m}{16}}
\frac{\eta(\frac{\tau}{2})}{\eta(\tau)^3} \cdot 
\frac{\vartheta_{01}(\tau,z)}{\vartheta_{11}(\tau,z)}
\\[2mm]
& &
\times \sum_{k=0}^{\frac{m_2}{2}-1}
q^{-\frac{1}{m}(k+\frac12)^2 + \frac12 (k+\frac12)} \, 
\big[\theta_{k+\frac12, \frac{m}{2}}-\theta_{-(k+\frac12), \frac{m}{2}}\big]
(2\tau, 2z)
\end{eqnarray*}}
proving 1). \, Proof of 2) is quite similar.
\end{proof}

\end{document}